 \documentclass[11pt,twoside,reqno,centertags,draft]{amsart}

 \PII{} 
\copyrightinfo{}{}

\thanks{AMS Subject Classifications:  35B65, 35J60, 35K55, 35R11}

 \usepackage{amsmath,amsthm,amsfonts,amssymb}
\usepackage[mathscr]{euscript}
 \usepackage{hyperref} 
 \usepackage{mathrsfs} 
\usepackage{graphicx}
\usepackage{color}
  \usepackage{epsfig}

\newcommand{\tl}{\tilde}

\newcommand{\e}{\epsilon}
\newcommand{\vep}{\varepsilon}
\newcommand{\vrh}{\varrho}
\newcommand{\pa} {\partial}
\newcommand{\el} {\ell}
\newcommand{\al} {\alpha}
\newcommand{\ba} {\beta}
\newcommand{\de} {\delta}
\newcommand{\ga} {\gamma}
\newcommand{\Ga} {\Gamma}
\newcommand{\Om} {\Omega}
\newcommand{\sg}{\sigma}

\newcommand{\De} {\Delta}
\newcommand{\la} {\lambda}

\newcommand{\ka}{\kappa}

\newcommand{\na} {\nabla}

\newcommand{\mb} {\mathbb}
\newcommand{\mc} {\mathcal}
\newcommand{\mf} {\mathfrak}

\newtheorem{Theorem}{Theorem}[section]
\newtheorem{Lemma}[Theorem]{Lemma}
\newtheorem{Proposition}[Theorem]{Proposition}
\newtheorem{Corollary}[Theorem]{Corollary}
\newtheorem{Remark}{Remark}
\newtheorem{Definition}{Definition}[section]

\def\Xint#1{\mathchoice
	{\XXint\displaystyle\textstyle{#1}}%
	{\XXint\textstyle\scriptstyle{#1}}%
	{\XXint\scriptstyle\scriptscriptstyle{#1}}%
	{\XXint\scriptscriptstyle\scriptscriptstyle{#1}}%
	\!\int}
\def\XXint#1#2#3{{\setbox0=\hbox{$#1{#2#3}{\int}$ }
		\vcenter{\hbox{$#2#3$ }}\kern-.6\wd0}}

\begin{document}
 
\title{ H\"older regularity results for 
parabolic nonlocal double phase problems }
\thanks{Accepted for publication: September 2023.} 
\date{}
\maketitle     
 
\vspace{ -1\baselineskip}

{\small
\begin{center}
 {\sc Jacques Giacomoni} \\
{{Universit\'e  de Pau et des Pays de l'Adour, LMAP (UMR E2S-UPPA CNRS 5142) }\\  Bat. IPRA, Avenue de l'Universit\'e F-64013 Pau, France }\\[10pt]
 {\sc Deepak Kumar} \\
{Research Institute of Mathematics, Seoul National University,}\\
{Seoul 08826, South Korea} \\[10pt]
 {\sc K. Sreenadh} \\
{Department of Mathematics, Indian Institute of Technology Delhi,}\\
{Hauz Khaz, New Delhi-110016, India } \\[10pt]
 (Submitted by: Giuseppe Mingione)  
\end{center}
}

\numberwithin{equation}{section}
\allowdisplaybreaks

 \smallskip

 \begin{quote}
\footnotesize
{\bf Abstract.}  
 In this article, we obtain higher H\"older regularity results for weak solutions to nonlocal problems driven by the fractional double phase operator
\begin{align*}
	\mc L u(x):=&2 \; {\rm P.V.} \int_{\mathbb R^N} \frac{|u(x)-u(y)|^{p-2}(u(x)-u(y))}{|x-y|^{N+ps_1}}dy \nonumber \\
	&+2 \; {\rm P.V.} \int_{\mathbb R^N} a(x,y) \frac{|u(x)-u(y)|^{q-2}(u(x)-u(y))}{|x-y|^{N+qs_2}}dy,
\end{align*}
where $1<p\leq q<\infty$,  $0<s_2, s_1<1$ and the modulating coefficient $a(\cdot,\cdot)$ is a non-negative bounded function. Specifically, we  prove higher space-time H\"older continuity result for weak solutions of time depending nonlocal double phase problems for a particular subclass of the modulating coefficients.  Using suitable approximation arguments, we further establish higher (global) H\"older continuity results for weak solutions to the stationary problems involving the operator $\mc L$ with modulating coefficients that are locally continuous. 
\end{quote}

\section{Introduction}
 In this article, our aim is to obtain space-time H\"older continuity results for the local weak solutions to the following nonlocal parabolic problem: 
\begin{equation*}
	\pa_t u+ \mc L u   = f \; \; \text{ in } \Om\times I, \tag{${\mc P}_t$}\label{probM}
\end{equation*}
where $\Om\subset\mb R^N$ ($N\geq 2$) is a bounded domain, $I:=(t_0,t_1]\subset\mb R$, $f\in L^\infty_{\mathrm{loc}}(\Om\times I)$ and the nonlocal operator $\mc L$ is defined as 
\begin{align*}
	\mc L u(x):=&2 \; {\rm P.V.} \int_{\mathbb R^N} \frac{|u(x)-u(y)|^{p-2}(u(x)-u(y))}{|x-y|^{N+ps_1}}dy \nonumber \\
	&+2 \; {\rm P.V.} \int_{\mathbb R^N} a(x,y) \frac{|u(x)-u(y)|^{q-2}(u(x)-u(y))}{|x-y|^{N+qs_2}}dy,
\end{align*}
with $1<p\leq q<\infty$,  $0<s_2, s_1<1$ and $a(\cdot,\cdot)\in L^\infty(\mb R^{N}\times\mb R^N)$ is a non-negative symmetric function.\par 
The operator $\mc L$ is known as the fractional double phase operator (as introduced in \cite{defillipPal}), which is a nonlocal analogue of the double phase operator $\mc D u:={\rm div}(|\na u|^{p-2}\na u+a(x)|\na u|^{q-2}\na u)$ with $1<p< q$ and $0\leq a(\cdot)$ bounded. The zero set of the modulating coefficient $a(\cdot)$ dictates the ellipticity of operator $\mc D$, that is, when $a(x)>0$, the problem has the $q$-Laplacian behaviour while on the zero set it has the $p$-Laplacian behaviour.
The corresponding energy functional falls into the category of the so-called functionals with nonstandard growth conditions of $(p, q)$-type, according to Marcellini's terminology \cite{marce2}. These kinds of functionals involve integrals of the form
\begin{align*}
	\mc E(u)=\int_{\Om} F(x,\na u(x))~dx,
\end{align*} 
where the energy density $F$, satisfies the unbalanced growth conditions
$$c_1|\xi|^p\leq |F(x,\xi)|\leq  c_2(|\xi|^q+1),\quad 1\leq p\leq q.$$
The physical significance of these models lies in the field of nonlinear elasticity, specifically in homogenization theory. A particular form of the functional $\mc E$ is the double phase energy functional given by 
$$
u\mapsto\int_{\Omega} (|\nabla u|^p+a(x)|\nabla u|^q)dx,\quad 0\leq a(x)\leq M,\ 1<p<q.
$$
This functional was first introduced by Zhikov in \cite{zhikov1}, to model the Lavrentiev phenomenon on strongly anisotropic materials. \par
The regularity for weak solutions to problems involving the operator $\mc D$ or the minimizers of corresponding functionals has been extensively studied in recent years, see \cite{baroni,byunoh,colombo,defilipis,JDS,marce2} and the  survey article \cite{mingioneJMAA}. Particularly, it is known that for $q\leq p+\ba$ (where $a(\cdot)\in C^{0,\ba}$), the gradient of the bounded weak solutions (or the minimizers) are locally H\"older continuous. 
Concerning the regularity results for time dependent problems involving nonstandard growth operators, we refer to \cite{defilipParGr,Giannetti}.  For the existence results of problems involving these kinds of operators, we mention \cite{arora,bogelein,JDSans,marce1,singer}.\par
In the stationary case and when $a\equiv 0$, the problem \eqref{probM} is driven by the fractional $p$-Laplacian ($(-\De)_p^s$), which is a nonlinear analogue of the well known fractional Laplacian. The study of problems involving fractional Sobolev spaces and corresponding nonlocal operators has gained ample attention due to their wide applications in the real world problems, such as game theory, image processing, anomalous diffusion, finance, conservation laws, phase transition,  obstacle problems  and material science, see \cite{nezzaH} and references therein.  The regularity results for weak solutions to problems involving these kinds of nonlocal operators (or minimizers of corresponding energy functional) are well known, see for instance \cite{brascoH,caffsilv,cozzi,dicastro,iann,nowak} and their references. Specifically, the weak solution is known to be $C^{0,s}$-regular up to the boundary, which is optimal.\par 
As far as nonlocal problems involving the fractional double phase operator are concerned, the corresponding regularity theory is a relatively new topic and has not been explored much barring a few works in the stationary case. In this direction, we mention \cite{defillipPal}, where De Filippis and Palatucci obtained an interior H\"older continuity result for viscosity solutions to the problem $\mc Lu=f\in L^\infty$. For the case of weak solutions, in \cite{fang}, Fang and Zhang proved a local H\"older continuity result (with unspecified H\"older exponent) of bounded weak solutions by using the De Giorgi-Nash-Moser type method, for $qs_2\leq ps_1$ and a non-negative bounded modulating coefficient. Here, additionally authors established that bounded weak solutions are viscosity solutions. Subsequently, by obtaining a suitable Logarithmic lemma and applying the De Giorgi-Nash-Moser iteration technique, Byun et al. in \cite{byun} obtained a similar local H\"older regularity result for the case $ps_1<qs_2$ when the modulating coefficient is assumed to be in the space $C^{0,\ba}(\mb R^N\times\mb R^N)$ with $qs_2-ps_1<\ba<1$. We also quote \cite{scott} for self improving properties of the bounded weak solutions to problems involving the operator $\mc L$, for the case $qs_2\leq ps_1$. For the constant modulating coefficient case, we mention the regularity results for weak solutions obtained in \cite{JDSacv,JDSjga,DDS}. Particularly, in \cite{JDSacv}, the almost optimal H\"older continuity result, up to the boundary, is proved in the super-quadratic case, while in \cite{JDSjga}, the global H\"older continuity result is established for all $p,q>1$. See also the recent work of Chaker et al. \cite{chaker} for the regularity results of nonlocal problems with nonstandard growth operators. However, our problem does not fall into the category of \cite{chaker} due to assumptions on the nonlinear growth and fractional differentiability orders.\par
The regularity theory for time dependent nonlocal problems has been developing significantly since the last decade. For instance,  \cite{caffChVs,kassmSch,chnag-L1} deals with the H\"older regularity of weak solutions by following different methods (weak Harnack inequality or De Giorgi's type approach) for the linear or non-degenerate nonlocal operator case. On the other hand, the development concerning the nonlinear nonlocal operators case is recent. In \cite{stromqvist}, St\"omqvist obtained a local boundedness result for parabolic problems for a large class of nonlocal problems whose prototype includes the fractional $p$-Laplacian for $p>2$. For similar problems and involving a non-zero reaction term, as a parabolic counterpart for \cite{dicastro}, Ding et al. in \cite{ding} proved the local boundedness property for all $p\geq 2N/(N+2s)$.  In continuation to \cite{brascoH}, Brasco et al. \cite{brascoP} obtained a higher interior space-time H\"older continuity result with precise exponents for the local weak solutions to fractional $p$-Laplacian parabolic problem in the super-quadratic case. Recently, in \cite{kadim,liao}, authors have obtained the local H\"older regularity results for weak solution to fractional $p$-Laplacian type problems, for all $p>1$, by means of suitable De Giorgi type iteration method and intrinsic scaling arguments. Concerning the nonstandard growth parabolic problems, in \cite{prasad}, authors have obtained a local boundedness result for variational solutions to a class of nonlocal parabolic double phase equations. \par
After these works the regularity results for parabolic problems driven by the fractional double phase operators were completely open. Addressing this issue, we establish space-time H\"older continuity result with specific exponents for the local weak solutions to problem \eqref{probM} when the modulating coefficient is locally translation invariant in $\Om\times\Om$.
To achieve this aim, we first obtain an iterative scheme involving the discrete differential of the solution with respect to the space variable using Moser's iteration technique much in the spirit of \cite{brascoP}.  Here,  we highlight that handling the terms involving the exponent $q$ is more tactical and tricky as compared to the constant modulating coefficient case (see \cite{JDSacv}).  
Subsequently, we obtain higher time regularity for weak solutions by using the higher space regularity estimates. In order to motivate our main regularity results, we sketch the proof of the existence of the weak solution to an initial value parabolic problem (valid for all $0<s_1,s_2<1$) in the Appendix. The novelty in this regard lies in introducing a new reflexive Banach space and proving fundamental properties such coercivity and continuity for the operator involved. Indeed, because of the varying ellipticity of the operator $\mc L$, it is not expected that the solution should lie on the fractional space $W^{s_1,p}\cap W^{s_2,q}$. Additionally, due to the absence of a suitable weighted fractional Poincar\'e type inequality with respect to the exponent $q$, we take into account appropriate fractional Musielak-Sobolev spaces and corresponding time dependent Banach space to define the notion of (local) weak solutions. \par 
Regarding the stationary case, the higher H\"older regularity for solutions to fractional double phase problems was not explored before. To this aim,  we establish H\"older continuity results with explicit H\"older exponent for weak solutions to problems involving a more general source term and a modulating coefficient that is continuous along the diagonal in $\Om\times\Om$. For this, we perform approximation arguments based on the methods presented in \cite{brascoH,caffsilv}. However, the lack of scaling property of the leading operators makes it difficult to pass the higher H\"older continuity results of problems with locally translation invariant modulating coefficients and homogeneous right hand side to our case. To overcome this, we choose the radius of the ball, where H\"older continuity is obtained, appropriately depending a priori on the data of the problem.
We emphasize that such results to the best of our knowledge were not known even for non-homogeneous operators where the modulating coefficients are bounded away from zero.    
Finally, we obtain a higher global H\"older continuity result for the weak solutions to the stationary problem when the right hand side is bounded, which is new to the current literature on the subject.\par
Turning to the layout of the rest of the paper, in Section 2, we mention the functional framework and our main results. Section 3 contains some technical results which will be needed in the subsequent sections. In Section 4, we establish higher space-time H\"older continuity result for locally translation invariant modulating coefficients. In Section 5, we obtain higher H\"older continuity results in the stationary case for more general modulating coefficients and source terms. Finally, in the appendix, we sketch the proof of the existence of weak solutions to the parabolic problem.

\section{Functions space framework and main results}
%We first fix some notations which will be used throughout the paper.
\subsection{Notations} 
For $1<q<\infty$, we set $[\xi]^{q-1}:=|\xi|^{q-2}\xi$, for all $\xi\in\mb R^N$ and we use the following notation for weighted nonlocal operators $$(-\De)_{q,a}^{s_2}u(x):=  2{\rm P.V.}\int_{\mb R^N} a(x,y) \frac{[u(x)-u(y)]^{q-1}}{|x-y|^{N+qs_2}}dy.$$
For $x_0\in\mb R^N$ and $v\in L^1(B_r(x_0))$, we set \[(v)_{x_0,r}:=\Xint-_{B_r(x_0)}v(x)dx=\frac{1}{|B_r(x_0)|}\int_{B_r(x_0)}v(x)dx\]
and when the center is clear from the context, we will write it as $(v)_r$.
We denote the parabolic cylinders by
\begin{align*}
	Q_{R,r}(x_0,t_0):=B_R(x_0)\times (t_0-r,t_0],
\end{align*}
where $B_R(x_0)\subset \mb R^N$ and $r>0$.
For simplicity in writing, we will often use the notation $(\el,s)\in\{(p,s_1),(q,s_2)\}$ together with $a_\el$, where $a_p\equiv 1$ and $a_q=a$ in $\mb R^N\times\mb R^N$, unless otherwise specified. We abbreviate the quantity $\|a\|_{L^\infty(\mb R^N\times\mb R^N)}$ by $\|a\|_{\infty}$. In several places, we will use 
\begin{align*}
	d\mu_{\el,s}(x,y):=\frac{dxdy}{|x-y|^{N+\el s}}, \quad\mbox{for }(\el,s)\in \{ (p,s_1), (q,s_2)\},
\end{align*} 
with $d\mu_1=d\mu_{p,s_1}$ and $d\mu_2=d\mu_{q,s_2}$. For $1<p<\infty$, we denote its H\"older conjugate by $p'$, i.e., $1/p+1/p'=1$. Furthermore, for $s\in (0,1)$, we denote the Critical Sobolev exponent by $p^*_{s}:=Np/(N-ps)$ if $N>ps$, otherwise it is an arbitrarily large number.
The constants $c$ and $C$ appearing in the proofs may vary line to line.
For a Banach space $(X,\|\cdot\|)$, we denote its topological dual by $X^*$.

\subsection{Functions spaces and definitions}%\label{secdefn}
In this subsection, we mention the function spaces used in the paper and define the notion of local weak solution to problem \eqref{probM}.
For $E\subset\mb R^N$, $x,y\in E$, $\xi\in\mb R$ and $u:E\to\mb R$, set
\begin{align*}
	& H_a(x,y,\xi):= \frac{|\xi|^p}{|x-y|^{ps_1}}+a(x,y)\frac{|\xi|^q}{|x-y|^{qs_2}} \quad\mbox{and }\\
	&\mc H_a(u,E\times E):=\iint_{E\times E}  H_a(x,y,|u(x)-u(y)|) \frac{dxdy}{|x-y|^N}.
\end{align*}
We define the following space:
\begin{align*}
	\mc W_a(\mb R^N):= \big\{ w\in L^p(\mb R^N) \ : \ \mc H_a(w,\mb R^N\times\mb R^N)<\infty \big\}
\end{align*}
endowed with the norm 
\begin{align*}
	\|u\|_{\mc W_a(\mb R^N)}:= \|u\|_{L^p(\mb R^N)}+[u]_{\mc W_a(\mb R^N)},
\end{align*}
where the Luxembourg seminorm is given by
\begin{align}\label{normLxm}
	[u]_{\mc W_a(\mb R^N)}:=\inf \big\{ \la>0 \ : \mc H_a\Big(\frac{u}{\la},\mb R^N\times\mb R^N\Big)\leq 1 \big\}.
\end{align}
It is a routine exercise to verify that $\mc W_a(\mb R^N)$ is a uniformly convex Banach space. Moreover, by the properties of modular spaces, it is easy to prove that 
\begin{itemize}
	\item[(i)] $\min\{ [u]_{\mc W_a(\mb R^N)}^{p}, [u]_{\mc W_a(\mb R^N)}^q \} \leq \mc H_a(u,\mb R^N\times\mb R^N)$ \\
	 $\leq \max\{ [u]_{\mc W_a(\mb R^N)}^{p}, [u]_{\mc W_a(\mb R^N)}^q \}$,
	\item[(ii)] $[u_n-u]_{\mc W_a(\mb R^N)}\to 0$ if and only if $\mc H_a(u_n-u,\mb R^N\times\mb R^N)\to 0$, as $n\to\infty$.
\end{itemize}
For any proper open  set $\Om\subset\mb R^N$, we define the space $\mc W_a(\Om)$ as below: 
\begin{equation*}
	\begin{aligned}
		\mc W_a(\Om) :=\big\{u\in L^p(\Om) \ : \ \int_{\Om}W_a(x) |u(x)|^q dx+\mc H_a(u,\Om\times\Om)<\infty  \big \},
	\end{aligned}
\end{equation*}
equipped with the norm 
\begin{align*}
	\|u\|_{\mc W_a(\Om)}:= \|u\|_{L^p(\Om)}+\|u\|_{L^q(W_a,\:\Om)}+[u]_{\mc W_a(\Om)},
\end{align*}
where
\begin{align}\label{eqw}
	\|u\|_{L^q(W_a,\:\Om)}^q:=\int_{\Om}W_a(x)|u(x)|^q dx \quad\mbox{with }W_a(x):=\int_{\mb R^N\setminus\Om}\frac{a(x,y)}{|x-y|^{N+qs_2}}dy,
\end{align}
and $[u]_{\mc W_a(\Om)}$ is defined as in \eqref{normLxm} by replacing $\mb R^N\times\mb R^N$ with $\Om\times\Om$. Note that the set of functions in $W^{s_1,p}(\Om)\cap W^{s_2,q}(\Om)$ with compact support is obviously contained in $\mc W_a(\Om)$. Moreover, it is also evident that $\mc W_a(\Om)$ is continuously embedded into $W^{s_1,p}(\Om)$. Here in after, we will suppress the subscript $a$ from the definitions of $\mc W_a$ and $W_a$, if it is clear from the context.  We call a function $u\in \mc W_{\rm loc}(\Om)$ if $u\in \mc W(\Om')$, for all $\Om'\Subset\Om$.\par 
Concerning the time dependent spaces, for $I\subset \mb R$ and a Banach space $X$ with a modular $\vrh$, we define 
\begin{align*}%\label{spctime}
	\mc Y(I;X):=\bigg\{u(t)\in X \mbox{ for a.e. }t\in I \mbox{ and }\int_{I} \vrh(u(t)) dt<\infty \bigg\}.
\end{align*}
In particular, for the space $\mc Y(I;\mc W(\Om))$, we take the modular $\vrh$ as $$\vrh(u(t)):=\mc H(u(\cdot,t),\Om\times\Om) +\int_{\Om}\big(|u(x,t)|^p+W(x)|u(x,t)|^q\big)dx$$ 
and equip it with the Luxembourg norm on $I$ similar to \eqref{normLxm}. 
%Similarly, we define $\mc Y(I;\mc W(\Om))$. 

\begin{Definition}%\label{defTail}
	Let $u:\mb R^N\to \mb R$ be a measurable function, $0<m,\sg<\infty$ and $a(\cdot,\cdot)\in L^\infty(\mb R^N\times\mb R^N)$ be a non-negative function. Then, we define the tail space as below:
	\begin{align*}
		L^{m}_{\sg,a}(\mb R^N) = \bigg\{ u\in L^{m}_{\rm loc}(\mb R^N) : \sup_{x\in\mb R^N}\int_{\mb R^N} a(x,y)\frac{|u(y)|^{m}}{(1+|y|)^{N+\sg}}dy <\infty \bigg\}
	\end{align*} 
	and for $a(\cdot,\cdot)\equiv 1$, we simply denote it by $L^m_{\sg}(\mb R^N)$.	For $m>1$, the nonlocal tail centered at $x_0\in\mb R^N$ with radius $R>0$ is defined as 
	\begin{align*}
		T_{\sg,a}^{m}(u;x_0,R)=R^{\sg}\sup_{x\in\mb R^N}\int_{\mb R^N\setminus B_R(x_0)} a(x,y)\frac{|u(y)|^{m-1}}{|x_0-y|^{N+\sg}} dy.
	\end{align*}
	Furthermore, for $J\subset\mb R$, we define the time depending nonlocal tail term by
	\begin{align*}
		T_{\infty,\sg,a}^{m}(u;x_0,R,J)=R^{\sg} \sup_{t\in J}\sup_{x\in\mb R^N} \int_{\mb R^N\setminus B_R(x_0)} a(x,y)\frac{|u(y,t)|^{m-1}}{|x_0-y|^{N+\sg}} dy.
	\end{align*}
	For $1<p,q<\infty$, we will often use the notation
	\begin{align*}
		T_{\infty,\sg}^{p,q}(u;x_0,R,J):=T_{\infty,\sg}^{p}(u;x_0,R,J)+T_{\infty,\sg}^{q}(u;x_0,R,J).	
	\end{align*}
\end{Definition}

\begin{Definition}[Local weak solution]
	Let $f\in (\mc Y(I;\mc W(\Om)))^*$, then we say that $u$ is a local weak solution to problem \eqref{probM}, if for all $J:=[T_0,T_1]\subset I$, 
	\begin{itemize}
		\item[(i)] $u\in \mc Y(J;\mc W_{\rm loc}(\Om))\cap C(J;L^2_{\rm loc}(\Om))\cap L^{p-1}(J;L^{p-1}_{ps_1}(\mb R^N))$\\ $\cap L^{q-1}(J;L^{q-1}_{qs_2,a}(\mb R^N))$ and
		\item[(ii)] for all $\phi\in \mc Y(J;\mc W(\Om))\cap C^1(J; L^2(\Om))$ with compact spatial support contained in $\Om$, there holds
		\begin{align}\label{eqWF}
			&-\int_{J}\int_{\Om} u(x,t)\pa_t \phi(x,t)dxdt \nonumber\\
			&\quad+ \int_{J} \iint_{\mb R^{2N}} \frac{[u(x,t)-u(y,t)]^{p-1}}{|x-y|^{N+s_1p}}(\phi(x,t)-\phi(y,t))dxdydt \nonumber\\
			&\quad+\int_{J} \iint_{\mb R^{2N}} a(x,y) \frac{[u(x,t)-u(y,t)]^{q-1}}{|x-y|^{N+s_2q}}(\phi(x,t)-\phi(y,t))dxdydt \nonumber \\
			&= \int_{\Om} u(x,T_0) \phi(x,T_0)dx- \int_{\Om} u(x,T_1) \phi(x,T_1)dx+  \langle f(\cdot,\cdot),\phi(\cdot,\cdot)\rangle_{\mc Y, \mc Y^*}.
		\end{align}
	\end{itemize}
	Local weak sub-solution (resp. super-solution) is defined similarly where \eqref{eqWF} holds by replacing the sign "$=$" with "$\leq$ (resp. $\geq$)" for all non-negative test functions. 
\end{Definition}
We say that a function $u$ is locally $\al$-H\"older continuous in space, denoted by $C_{x,{\rm loc}}^{0,\al}(\Om\times I)$, if for any $K\times J\Subset \Om\times I$, 
\begin{align*}
	\sup_{t\in J}\; [u(\cdot,t)]_{C^{0,\al}(K)}<\infty
\end{align*}
and analogously the local $\ga$-H\"older continuity in time is defined.\par
Finally, we recall the following standard inequality:  for $\ell\geq 2$, there exists a constant $c(\ell)>0$ such that 
\begin{align}\label{eqmon}
	c(\ell)|\xi-\zeta|^\ell \leq ( [\xi]^{\ell-1}- [\zeta]^{\ell-1}) (\xi-\zeta)  \quad\mbox{for all }\xi,\zeta\in\mathbb R.
\end{align}

\subsection{Assumptions and main results}
We establish a higher H\"older continuity result for locally translation invariant modulating coefficients. Precisely, we assume the following assumption:
\begin{itemize}
	\item[\textbf{(A.1)}]The function $a(\cdot,\cdot)$ is locally translation invariant in $\Om\times\Om$, that is, for any $K_1\times K_2\Subset\Om\times\Om$, there holds $a(x+h,y+h)=a(x,y)$ for all  $(x,y)\in K_1\times K_2$ and for all $|h|<\min\{{\rm dist}(K_1,\Om), {\rm dist}(K_2,\Om)\}$.
\end{itemize}
%\begin{assumption}\label{A.2}
% The function $a(\cdot,\cdot)$ is locally translation invariant in $\Om\times\Om$, that is, for any $K_1\times K_2\Subset\Om\times\Om$, there holds $a(x+h,y+h)=a(x,y)$ for all  $(x,y)\in K_1\times K_2$ and for all $|h|<\min\{{\rm dist}(K_1,\Om), {\rm dist}(K_2,\Om)\}$. 
%\end{assumption}
Before stating our main result in this regard, we fix the following notations:
\begin{align}\label{eqTheta}
	\Theta\equiv \Theta(p,s_1):=\min\bigg\{1,\frac{ps_1}{p-1} \bigg \}
\end{align}
and \begin{align}\label{eqGamma}
	\Ga\equiv \Ga(p,s_1):=\begin{cases}
		\frac{1}{ps_1-(p-2)} &\mbox{if }s_1\geq \frac{p-1}{p},\\
		1 &\mbox{if }s_1< \frac{p-1}{p}.
	\end{cases}
\end{align}
\begin{Theorem}\label{thminh}
 Suppose that $p\geq 2$ and $qs_2\leq ps_1$. Let the assumption {\rm \textbf{(A.1)}} holds and let $u$ be a local weak solution to problem \eqref{probM} such that $u\in L^\infty_{\rm loc}(I; L^\infty(\mb R^N))$. 
 Then, for all $\al\in (0,\Theta)$ and all $\ga\in (0,\Ga)$, 
	\begin{align*}
		u\in C^{0,\al}_{x, {\rm loc}}(\Om\times I)\cap C^{0,\ga}_{t, {\rm loc}}(\Om\times I),
	\end{align*} 
 where $\Theta$ and $\Gamma$ are given by \eqref{eqTheta} and \eqref{eqGamma}, respectively.\\
 Particularly, for $Q_{2R,2R^{s_1p}}\equiv Q_{2R,2R^{s_1p}}(x_0,\hat t_0)\Subset\Om\times I$ and for all $\al\in (0,\Theta)$ and $\ga\in (0,\Ga)$, there exists a positive constant $C$ depending only on $N,s_1,s_2,p,q,\al,\ga$ and $\|a\|_\infty$ such that 
	\begin{align*}
		&|u(x,t)-u(y,\tau)| \nonumber\\
		&\leq  \frac{C}{R^{\al}} \Big(\|u\|_{L^\infty(\mb R^N\times[\hat t_0-R^{s_1p},\hat t_0])}+1+\|f\|^\frac{1}{p-1}_{L^\infty(Q_{R,R^{s_1p}})}\Big)^\frac{qj_\infty}{p} |x-y|^\al \\ 
		&+\frac{C}{R^{ps_1\ga}} \Big(\|u\|_{L^\infty(\mb R^N\times[\hat t_0-R^{s_1p},\hat t_0])}+1+\|f\|^\frac{1}{p-1}_{L^\infty(Q_{R,R^{s_1p}})}\Big)^\frac{q(\ga(p-2)+1)j_\infty}{p}|t-\tau|^\ga, 
	\end{align*}
 for all $x,y\in B_\frac{R}{4}(x_0)$ and $t,\tau\in (\hat t_0-\frac{R^{ps_1}}{4},\hat t_0]$, where  $j_\infty\in\mb N$ depends only on $N,p,s_1$ and $\al$.
\end{Theorem}

\begin{Remark}
	We remark that Theorem \ref{thminh} is true for local weak solutions that are not bounded in the whole $\mb R^N$ with respect to the space variable. In this regard, we need a weaker condition, namely $u\in L^\infty_{{\rm loc}}(\Om\times I)\cap L^{\infty}_{\rm loc}(I;L^{p-1}_{ps_1}(\mb R^N)) \cap L^{\infty}_{\rm loc}(I;L^{q-1}_{qs_2}(\mb R^N))$. However, in this case, the value of   $\Ga$ (the exponent concerning the time regularity) is reduced and given by $\Ga:=\min\{1,1/(ps_1)\}$ (see Theorem \ref{thminwtl} and Remark \ref{remwtglbd} for further details).
\end{Remark}
Next, we consider the following stationary counterpart of problem \eqref{probM}:
\begin{equation*}
	\mc L u   = f \; \; \text{ in } \Om, \tag{${\mc S}$}\label{probSt}
\end{equation*}
where $f\in L^{\Bbbk}_{\rm loc}(\Om)$, for $\Bbbk>\max\{N/(ps_1),1\}$.
\begin{Remark}\label{rem.2.doubl.phas}
 For the case $ps_1<qs_2<ps_1+2(1-s_1)$, we point out that under the assumption {\rm \textbf{(A.1)}}, with Proposition \ref{propin1} in our hand, we obtain that estimates \eqref{eqi50} holds for a local weak solution $u$ to problem \eqref{probM} in this case, too. This in turns implies that $u\in C^{0,\al}_{x,\rm loc}(\Om\times I)$ for all $\al<s_1$. In particular, for the stationary problem \eqref{probSt}, under the assumption {\rm \textbf{(A.1)}}, we have that $u\in C^{0,\al}_{\rm loc}(\Omega)$ for all $\al<s_1$, whenever $qs_2<ps_1+2(1-s_1)$. On the other hand, for the case $qs_2\leq ps_1$, a stationary counterpart of Theorem \ref{thminh} holds for all $\al<\Theta$ with $\Theta$ given by \eqref{eqTheta}. 
\end{Remark} 
In the result below, we remove the restriction of local translation invariant nature of $a(\cdot,\cdot)$ in $\Om\times\Om$ by the following: 
\begin{itemize}
	\item[\textbf{(A.2)}] The function $a(\cdot,\cdot)$ is continuous along the diagonal in $\Om\times\Om$, that is, for any $R>0$ with $B_R(x_0)\Subset\Om$, there exists $\e>0$ such that for all $\de>0$, there exists $h_\de>0$ satisfying 
	\begin{align*}
		\sup_{\substack{x,y\in B_R(x_0)\\ |x-y|\leq \e}} |a(x+h,y+h)-a(x,y)|\leq \de \quad\mbox{for all }|h|<{h_\de}.
	\end{align*}
\end{itemize}
\begin{Theorem}\label{thmingenk}
 Suppose that $2\leq p\leq q<\min\{p^*_{s_1}, ps_1/s_2\}$. Let $a(\cdot,\cdot)$ satisfies the assumption {\rm\textbf{(A.2)}} and let $u$ be a local weak solution to problem \eqref{probSt} such that $u\in  L^{q-1}_{qs_2}(\mb R^N)$. Then, $u\in C^{0,\al}_{\rm loc}(\Om)$ for all $\al\in (0,{\Theta_1})$, where 
	\begin{align}\label{eqThetaiot}
		{\Theta_1}:=\min\bigg\{ \frac{ps_1-N/\Bbbk}{p-1},\frac{qs_2}{q-1},1\bigg\}.
	\end{align}
	% More precisely, for any $B_{2R}\Subset\Om$ with $R\in (0,1)$ and for all $\al\in (0,{\Theta_1})$, the quantity $[u]_{C^{\al}(B_{R/4})}$ depends only on $N,s_1,p,q,\al$, $\|a\|_\infty$ and as a non-decreasing function of $\big( \|u\|_{L^\infty(B_{R/2})}+1+\|f\|^\frac{1}{p-1}_{L^\Bbbk(B_{R/2})}+\sum_{(\el,s)}T^{\el}_{s\el}(u;x_0,R/2)^\frac{1}{\el-1}
	% \big)$.
	% . such that 
	%	 \begin{align}\label{cont bd st}
		%	[u]_{C^{\al}(B_{R/4})}\leq  \frac{C}{R^{\al}} \Big( \|u\|_{L^\infty(B_{R/2})}+1+\|f\|^\frac{1}{p-1}_{L^\Bbbk(B_{R/2})}+\sum_{(\el,s)}T^{\el}_{s\el}(u;x_0,R/2)^\frac{1}{\el-1}
		%	\Big).
		% \end{align}
\end{Theorem}
Regarding the global H\"older continuity result, we have the following.
\begin{Remark}
 Let us consider the principal nonlocal operator as $(-\De)_{p,b}^{s_1} + (-\De)_{q,a}^{s_2}$, where the new coefficient function $b$ is such that $0<\Lambda^{-1}\leq b(x,y)\leq \Lambda$, for all $x,y\in\mb R^N$. Then, under the same vanishing assumption as that of the coefficient function $a$, we remark that by following the same argument as in this paper, all the results of Theorem \ref{thminh}, Theorem \ref{thmingenk} and Remark \ref{rem.2.doubl.phas} can be extended to the case of problems \eqref{probM} and \eqref{probSt} with the aforementioned operator. 
 %See \cite{byu-kim-kum} for a similar problem and some regularity properties for the corresponding weak solutions.
\end{Remark} 
\begin{Corollary}\label{corbdry}
 Let $\Om$ be a bounded domain with $C^{1,1}$ boundary and let the hypotheses of Theorem \ref{thmingenk} hold. Additionally, if $s_1<q's_2$, assume that  $a(\cdot,\cdot)\in C^{0,\ba}(\mb R^{2N})$ with $\ba>s_2$. Let $u\in\mc W(\mb R^N)$ be a weak solution to problem \eqref{probSt} with $f\in L^\infty(\Om)$ and the homogeneous Dirichlet boundary condition in $\mb R^N\setminus\Om$. Then, $u\in C^{0,\al}(\mb R^N)$, for all $\al<s_1$.
\end{Corollary}

\section{Some technical results}
%In this section, we mention some of the auxiliary results.
\subsection{The case of parabolic problems}
We first define the time regularization of the test functions as below.
\begin{Definition}[Regularization of test functions]\label{regtst}
	Let $\varsigma:\mb R\to \mb R$ be a non-negative, even smooth function with compact support in $(-1/2,1/2)$ satisfying $\int_{\mb R}\varsigma(\tau)d\tau=1$. For $g\in L^1((a,b))$, we define
	\begin{align*}%\label{eqregtst}
		g^{\vep}(t)=\frac{1}{\vep} \int_{t-\vep/2}^{t+\vep/2}\varsigma\Big(\frac{t-\sg}{\vep}\Big)g(\sg)d\sg= \frac{1}{\vep} \int_{-\vep/2}^{\vep/2}\varsigma\Big(\frac{\sg}{\vep}\Big)g(t-\sg)d\sg,
	\end{align*}
	for $t\in (a,b)$ with $\vep\in (0,\min\{b-t,t-a\})$. 
\end{Definition}
\begin{Lemma}[Caccioppoli inequality]\label{cacciopp} Let $u$ be a local sub-solution to problem \eqref{probM} such that $u\in L^\infty_{{\rm loc}}(\Om\times I)\cap L^{\infty}_{\rm loc}(I;L^{p-1}_{ps_1}(\mb R^N)) \cap L^{\infty}_{\rm loc}(I;L^{q-1}_{qs_2,a}(\mb R^N))$. Suppose that $B_r\equiv B_r(x_0)\Subset \Om$, for some $r>0$ and $t_2<t_3$, $\tau>0$ satisfies $[t_2-\tau,t_3]\subset I:=(t_0,t_1]$. Then, for all non-negative functions $\psi\in C_c^\infty(B_r)$ and $\eta\in C^\infty(\mb R)$ satisfying $\eta\equiv 0$ if $t\leq t_2-\tau$ and $\eta(t)\equiv 1$ if $t\geq t_2$, there exists a constant $C=C(N,s_1,s_2,p,q,\|a\|_{\infty})>0$ such that   
	\begin{align*}
		&\sup_{t_2<t<t_3}\int_{B_r}w_+^2(x,t)\psi^q(x)dx \nonumber\\
		&\,\,+\int_{t_2-\tau}^{t_3}\int_{B_r} \int_{B_r} \frac{|w_+(x,t)\psi^\frac{q}{p}(x)-w_+(y,t)\psi^\frac{q}{p}(y)|^p}{|x-y|^{N+s_1p}}\eta^2(t)dxdydt \nonumber\\ 
		&\,\,+\int_{t_2-\tau}^{t_3}\int_{B_r} \int_{B_r}a(x,y) \frac{|w_+(x,t)\psi(x)-w_+(y,t)\psi(y)|^q}{|x-y|^{N+s_2q}}\eta^2(t)dxdydt \nonumber\\
		&\leq C \sum_{(\el,s)} \int_{t_2-\tau}^{t_3}\int_{B_r}\int_{B_r} a_\el(x,y) \frac{|\psi(x)-\psi(y)|^\el}{|x-y|^{N+s\el}}\\
		&\qquad\qquad\times(w_+(x,t)+w_+(y,t))^\el\eta^2(t) dxdydt \nonumber\\
		&\ + C   \sum_{(\el,s)} \sup_{\substack{t_2-\tau<t<t_3 \\ x\in {\rm supp}(\psi)}} \Bigg( \int_{\mb R^N\setminus B_r}a_\el(x,y) \frac{w_+(y,t)^{\el-1}}{|x-y|^{N+s\el}}dy\Bigg) \nonumber\\ &\quad\qquad\times\int_{t_2-\tau}^{t_3}\int_{B_r}w_+(x,t)\psi^q(x)\eta^2(t)dxdt \nonumber\\
		& \ + C \int_{t_2-\tau}^{t_3}\int_{B_r} |f(x,t)| w_+(x,t)\psi^q(x)\eta^2(t)dxdt \nonumber\\
		& \,+C \int_{t_2-\tau}^{t_3} \int_{B_r}w_+^2(x,t)\psi^q(x)\eta(t)|\pa_t\eta(t)|dxdt,
	\end{align*}  
	where $w_+=\max\{u-k,0\}$, for $k\in\mb R$.
\end{Lemma}
\begin{proof}
	First we assume $t_3<t_1$ and take 
	\begin{align*}
		0<\e<\frac{\e_0}{2}:=\frac{1}{4}\min\{t_2-\tau-t_0,t_1-t_3,t_3-t_2+\tau\}.
	\end{align*}
	Set $\hat t_2:=t_2-\tau-\e_0$ and $\hat t_3:=t_3+\e_0$.
	Let $\phi\in \mc Y((\hat t_2,\hat t_3);\mc W(B_r))\cap C^1((\hat t_2,\hat t_3); L^2(B_r))$ be such that its spatial support is compactly contained in $B_r$. Let $\phi^\e$ be the regularization of $\phi$, as defined in Definition \ref{regtst}. Thus, from \eqref{eqWF} (for sub-solution), we have 
	\begin{align}\label{eq20}
		&-\int_{J}\int_{B_r} u(x,t)\pa_t \phi^\e(x,t)dxdt \nonumber\\
		&\quad+ \int_{J} \iint_{\mb R^{2N}} \frac{[u(x,t)-u(y,t)]^{p-1}}{|x-y|^{N+s_1p}}(\phi^\e(x,t)-\phi^\e(y,t))dxdydt \nonumber\\
		&\quad+\int_{J} \iint_{\mb R^{2N}}a(x,y) \frac{[u(x,t)-u(y,t)]^{q-1}}{|x-y|^{N+s_2q}}(\phi^\e(x,t)-\phi^\e(y,t))dxdydt \nonumber \\
		&\leq \int_{B_r} u(x,T_2) \phi^\e(x,T_2)dx- \int_{B_r} u(x,T_3) \phi^\e(x,T_3)dx \nonumber\\
		&\quad+  \int_{J} \int_{B_r}f(x,t)\phi^\e(x,t)dx dt,
	\end{align} 
	where $J=[T_2,T_3]$ with $T_2=t_2-\tau$ and $T_3\in (t_2,t_3]$ to be determined later (note that $T_2-\e$ and $T_3+\e\in (\hat t_2,\hat t_3)$ for all $\e<\e_0/2$). Then, by the properties of convolution, Fubini's Theorem and integration by parts, it follows that
	\begin{align*}
		&-\int_{J}\int_{B_r} u(x,t)\pa_t \phi^\e(x,t)dxdt\\
		&= \int_{T_2+\frac{\e}{2}}^{T_3-\frac{\e}{2}} \int_{B_r}\pa_t u^\e(x,t)\phi(x,t)dxdt+ \Upsilon(\e) \\
		&\ - \int_{B_r} u^\e(x,T_3-\frac{\e}{2}) \phi(x,T_3-\frac{\e}{2})dx+ \int_{B_r} u^\e(x,T_2+\frac{\e}{2}) \phi(x,T_2+\frac{\e}{2})dx,
	\end{align*}
	where 
	\begin{align*}
		\Upsilon(\e):=&-\int_{B_r} \int_{T_2-\frac{\e}{2}}^{T_2+\frac{\e}{2}}\Bigg(\frac{1}{\e}\int_{T_2}^{\sg+\frac{\e}{2}}u(x,t)\varsigma(\frac{\sg-t}{\e})dt \Bigg)\pa_\sg\phi(x,\sg)d\sg dx \\
		&-\int_{B_r} \int_{T_3-\frac{\e}{2}}^{T_3+\frac{\e}{2}}\Bigg(\frac{1}{\e}\int_{\sg-\frac{\e}{2}}^{T_3}u(x,t)\varsigma(\frac{\sg-t}{\e})dt \Bigg)\pa_\sg\phi(x,\sg)d\sg dx.
	\end{align*}
	Using this in \eqref{eq20}, we get 
	\begin{align*}
		&\int_{T_2+\frac{\e}{2}}^{T_3-\frac{\e}{2}} \int_{B_r}\pa_t u^\e(x,t)\phi(x,t)dxdt \nonumber\\
		&+ \int_{J} \iint_{\mb R^{2N}} \frac{[u(x,t)-u(y,t)]^{p-1}}{|x-y|^{N+s_1p}}(\phi^\e(x,t)-\phi^\e(y,t))dxdydt \nonumber\\
		&+\int_{J} \iint_{\mb R^{2N}}a(x,y) \frac{[u(x,t)-u(y,t)]^{q-1}}{|x-y|^{N+s_2q}}(\phi^\e(x,t)-\phi^\e(y,t))dxdydt + \Upsilon(\e)\nonumber \\
		&\leq \int_{B_r} u(x,T_2) \phi^\e(x,T_2)dx- \int_{B_r} u(x,T_3) \phi^\e(x,T_3)dx \nonumber\\
		& \ +  \int_{J} \int_{B_r}f(x,t)\phi^\e(x,t)dx dt 
		 +\int_{B_r} u^\e(x,T_3-\frac{\e}{2}) \phi(x,T_3-\frac{\e}{2})dx \nonumber\\
		 &\ - \int_{B_r} u^\e(x,T_2+\frac{\e}{2}) \phi(x,T_2+\frac{\e}{2})dx.
	\end{align*}
	We take $\phi(x,t)=v^\e(x,t)\psi^q(x)\eta^2(t)$ in the above equation, where $v(x,t):=(u^\e-k)_+(x,t)$. Thus, for $(\el,s)\in \{(p,s_1),(q,s_2)\}$,
	\begin{align*}
		&\underbrace{\int_{T_2+\frac{\e}{2}}^{T_3-\frac{\e}{2}} \int_{B_r}\pa_t (u^\e(x,t)-k)_+(v^\e\psi^q\eta^2)(x,t)dxdt}_{\mf I_1^\e}+ \Upsilon(\e) \nonumber\\
		&+\sum_{(\el,s)} \underbrace{\begin{aligned} \int_{J}\int_{B_r}\int_{B_r} &a_\el(x,y) \frac{[w(x,t)-w(y,t)]^{\el-1}}{|x-y|^{N+s\el}}\\
		 &\times((v^\e\psi^q\eta^2)^\e(x,t)-(v^\e\psi^q\eta^2)^\e(y,t))dxdydt\end{aligned}}_{\mf I_2^\e(\el)} \nonumber\\
		&+2\sum_{(\el,s)} \underbrace{\int_{J} \int_{\mb R^N\setminus B_r}\int_{B_r} a_\el(x,y) \frac{[w(x,t)-w(y,t)]^{\el-1}}{|x-y|^{N+s\el}}(v^\e\psi^q\eta^2)^\e(x,t)dxdydt}_{\mf I_3^\e(\el)} \\
		&\leq  \underbrace{\int_{J} \int_{B_r}f(x,t)(v^\e\psi^q\eta^2)^\e(x,t)dx dt}_{\mf{I}_4^\e} \nonumber\\
		& + \underbrace{\int_{B_r} u^\e(x,T_3-\frac{\e}{2}) (v^\e\psi^q\eta^2)(x,T_3-\frac{\e}{2})dx - \int_{B_r} u(x,T_3) (v^\e\psi^q\eta^2)^\e(x,T_3)dx}_{\mf{I}_5^\e} \nonumber\\
		& +\underbrace{\int_{B_r} u(x,T_2) (v^\e\psi^q\eta^2)^\e(x,T_2)dx- \int_{B_r} u^\e(x,T_2+\frac{\e}{2}) (v^\e\psi^q\eta^2)(x,T_2+\frac{\e}{2})dx}_{\mf{I}_6^\e}.
	\end{align*}
	An integration by parts formula yields
	\begin{align*}
		\mf I_1^\e=& \frac{1}{2}\int_{B_r}(u^\e(x,T_3-\frac{\e}{2})-k)_+^2 \psi^q(x)\eta^2(T_3-\frac{\e}{2})dx \\
		&- \frac{1}{2} \int_{B_r}(u^\e(x,T_2+\frac{\e}{2})-k)_+^2 \psi^q(x)\eta^2(T_2+\frac{\e}{2})dx \\
		&-\int_{T_2+\frac{\e}{2}}^{T_3-\frac{\e}{2}} \int_{B_r} (u^\e(x,t)-k)_+^2\psi^q(x)\eta(t)\pa_t\eta(t)dxdt.
	\end{align*}
	Noting $u\in C([\hat t_2,\hat t_3];L^2(B_r))$, we pass through the limit as $\e\to 0$ in the above equation to get 
	\begin{align*}
		\mf I_1^\e\to &\mf I_1\nonumber\\
		&:=\frac{1}{2}\int_{B_r}w_+(x,T_3)^2 \psi^q(x)\eta^2(T_3)dx 
		- \frac{1}{2} \int_{B_r}w_+(x,T_2)^2 \psi^q(x)\eta^2(T_2)dx \nonumber\\
		&\quad-\int_{J} \int_{B_r} w_+(x,t)^2\psi^q(x)\eta(t)\pa_t\eta(t)dxdt.
	\end{align*}
	To estimate $\mf I_2^\e(\el)$, using some elementary properties of convolution, we can prove that the sequence $\{ (v^\e\eta^2)^\e\psi^q\}_{\e}$ is bounded in $\mc Y(J;\mc W(B_r))$. Indeed,
	\begin{align*}
		\int_{J} [(v^\e\eta^2)^\e\psi^q]^\el_{W^{s,\el}_{a_\el}(B_r)}dt &\leq C\|\psi\|_{L^\infty}^{q\el}\int_{J} [(v^\e\eta^2)^\e]^\el_{W^{s,\el}_{a_\el}(B_r)}dt \\
		&\quad+ C\|\na\psi\|^{\el}_{L^\infty} \|\psi\|^{(q-1)\el}_{L^\infty} \int_{J} \|(v^\e\eta^2)^\e\|^\el_{L^{\el}(a_\el,B_r)}dt
	\end{align*} 
	and then we use the uniform $L^\infty$ bound on $v^\e$.
	Therefore, by the uniform convexity of the space $\mc Y(J;\mc W(B_r))$, up to a subsequence, we get
	\begin{align*}
		(v^\e\psi^q\eta^2)^\e \rightharpoonup (u-k)_+\psi^q\eta^2 \quad\mbox{weakly in }\mc Y(J;\mc W(B_r)),
	\end{align*}
	which implies that
	\begin{align*}
		a_\el(x,y)^{1/\el}\frac{(v^\e\psi^q\eta^2)^\e(x,t)-(v^\e\psi^q\eta^2)^\e(y,t)}{|x-y|^{N/\el+s}}
	\end{align*}
	converges weakly in $L^\el(J;L^\el(B_r\times B_r))$. Moreover, on account of $u\in \mc Y(J;\mc W(B_r))$, we have
	\begin{align*}
		a_\el^{1-1/\el}(x,y)[w(x,t)-w(y,t)]^{\el-1} |x-y|^{1-1/\el}\in L^\frac{\el}{\el-1}(J;L^\frac{\el}{\el-1}(B_r\times B_r)).
	\end{align*}
	Thus, we get that $\mf I_2^\e(\el)\to \mf I_2(\el)$, as $\e\to 0$, where
	\begin{align*}
		\mf I_2(\el)
		&= \int_{J}\int_{B_r}\int_{B_r} a_\el(x,y) \frac{[w(x,t)-w(y,t)]^{\el-1}}{|x-y|^{N+s\el}}\\
		&\qquad\quad\times(w_+(x,t)\psi^q(x)-w_+(y,t)\psi^q(y))\eta^2(t)dxdydt.
	\end{align*}
	For $\mf I_3^\e$, we proceed similarly by taking into account boundedness of the sequence $\{ (v^\e\eta^2)^\e\psi^q\}_{\e}$ in the space $L^{p-1}(J;L^{p-1}_{ps_1}(\mb R^N)) \cap L^{q-1}(J;L^{q-1}_{qs_2,a}(\mb R^N))$ instead of $\mc Y(J;\mc W(B_r))$. Consequently,
	\begin{align*}
		\mf I_3^\e(\el)\to \mf I_3(\el) 
		:=\int_{J}\int_{\mb R^N\setminus B_r}\int_{B_r}& a_\el(x,y) \frac{[w(x,t)-w(y,t)]^{\el-1}}{|x-y|^{N+s\el}}\\
		&\times w_+(x,t)\psi^q(x)\eta^2(t)dxdydt.
	\end{align*}
	By using similar reasoning, we can prove that $\Upsilon(\e)\to 0$, as $\e\to 0$. Further, on account of properties of convolution, we get $\mf I_5^\e\to 0$ and $\mf I_6^\e\to 0$ as $\e\to 0$. Additionally, 
	\begin{align*}
		\mf I_4^\e\to \mf I_4:=\int_{J} \int_{B_r}f(x,t)w_+(x,t)\psi^q(x)\eta^2(t)dx dt.
	\end{align*}
	Therefore,
	\begin{align}\label{eq21}
		\mf I_1+\mf I_2+\mf I_3\leq \mf I_4.
	\end{align}
	\textit{Step 1}: Estimate of $\mf I_1$. Noting the assumption that $\eta(T_2)=0$ and $\eta\equiv 1$ in $[t_2,t_3]$, we get 
	\begin{align}\label{eq22}
		\mf I_1=\frac{1}{2}\int_{B_r}w_+(x,T_3)^2 \psi^q(x)dx 
		-\int_{T_2}^{T_3} \int_{B_r} w_+(x,t)^2\psi^q(x)\eta(t)\pa_t\eta(t)dxdt.
	\end{align}
	\textit{Step 2}: Estimate of $\mf I_2(\el)$. Setting $\psi_\el=\psi^{q/\el}$, we proceed exactly as in \cite[Lemma 3.2]{JDSacv} to get 
	\begin{align}\label{eq23}
		&\mf I_2(\el)\nonumber\\
		&\geq \frac{1}{4}\int_{T_2}^{T_3}\int_{B_r} \int_{B_r}a_\el(x,y) \frac{|w_+(x,t)-w_+(y,t)|^\el}{|x-y|^{N+s\el}}\eta^2(t)(\psi_\el(x)+\psi_\el(y))^\el dxdydt \nonumber\\
		& \ -C\int_{T_2}^{T_3}\int_{B_r} \int_{B_r}a_\el(x,y) \frac{|\psi_\el(x)-\psi_\el(y)|^\el}{|x-y|^{N+s\el}}(w_+(x,t)+w_+(y,t))^\el\eta^2(t)dxdydt.
	\end{align}
	\textit{Step 3}: Estimate of $\mf I_3(\el)$. Observe that
	\begin{align*}
		[w(x,t)-w(y,t)]^{\el-1}w_+(x,t)\geq -w_+(y,t)^{\el-1}w_+(x,t).
	\end{align*}
	Thus, from the definition of $T_2$, we deduce that 
	\begin{align}\label{eq24}
		\mf I_3(\el)\geq &-C  \sup_{\substack{t_2-\tau<t<T_3 \\ x\in supp(\psi)}} \Bigg( \int_{\mb R^N\setminus B_r}a_\el(x,y) \frac{w_+(y,t)^{\el-1}dy}{|x-y|^{N+s\el}}\Bigg) \nonumber\\
		&\quad\times \int_{t_2-\tau}^{T_3}\int_{B_r}w_+(x,t)\psi^p(x)\eta^2(t)dxdt.
	\end{align}
	Now, 
	noting 
	\begin{align*}
		&|w_+(x,t)\psi_\el(x)-w_+(y,t)\psi_\el(y)|^\el \nonumber\\
		&\leq 2^{\el-1}|w_+(x,t)-w_+(y,t)|^\el (\psi_\el(x)^\el+\psi_\el(y)^\el) \\
		&\ +2^{\el-1}(w_+(x,t)^\el+w_+(y,t)^\el) |\psi_\el(x)-\psi_\el(y)|^\el
	\end{align*}
	and combining \eqref{eq22}, \eqref{eq23} and \eqref{eq24} in \eqref{eq21}, we get 
	\begin{align}\label{eq30}
		&\int_{B_r}w_+^2(x,T_3)\psi^p(x)dx \nonumber\\
		&+\int_{t_2-\tau}^{T_3}\int_{B_r} \int_{B_r} \frac{|w_+(x,t)\psi_p(x)-w_+(y,t)\psi_p(y)|^p}{|x-y|^{N+s_1p}}\eta^2(t)dxdydt \nonumber\\ 
		&+\int_{t_2-\tau}^{T_3}\int_{B_r} \int_{B_r}a(x,y) \frac{|w_+(x,t)\psi(x)-w_+(y,t)\psi(y)|^q}{|x-y|^{N+s_2q}}\eta^2(t)dxdydt \nonumber\\
		&\leq C \sum_{(\el,s)} \int_{t_2-\tau}^{t_3}\int_{B_r}\int_{B_r} a_\el(x,y) \frac{|\psi_\el(x)-\psi_\el(y)|^\el}{|x-y|^{N+s\el}}(w_+(x,t)+w_+(y,t))^\el dxdydt \nonumber\\
		&\ + C   \sum_{(\el,s)} \sup_{\substack{t_2-\tau<t<t_3 \\ x\in supp(\psi)}} \Bigg( \int_{\mb R^N\setminus B_r}a_\el(x,y) \frac{w_+(y,t)^{\el-1}dy}{|x-y|^{N+s\el}}\Bigg) \nonumber\\
	   &\qquad\qquad\times \int_{t_2-\tau}^{t_3}\int_{B_r}w_+(x,t)\psi^q(x)\eta^2(t)dxdt \nonumber\\
		& \ + C \int_{t_2-\tau}^{t_3}\int_{B_r} |f(x,t)| w_+(x,t)\psi^q(x)\eta^2(t)dxdt \nonumber\\
		& \ +C \int_{t_2-\tau}^{t_3} \int_{B_r}w_+^2(x,t)\psi^q(x)\eta(t)|\pa_t\eta(t)|dxdt.
	\end{align}
	To get the required result of the Lemma, we take $T_3\in (t_2,t_3]$  with the property
	\begin{align*}
		\int_{B_r}w_+^2(x,T_3)\psi^p(x)dx=\sup_{t_2<t<t_3} \int_{B_r}w_+^2(x,t)\psi^p(x)dx
	\end{align*}
	and once again we take $T_3=t_3$,  then we sum the resulting inequalities. This proves the Lemma for $t_3<t_1$. Note that the constant $C$ appearing in \eqref{eq30} does not depend on $t_3$. Therefore, by following an approximation argument much in the spirit of \cite[Proposition 4.1, Step 7]{brascoP}, we get a similar result for the case $t_3=t_1$. This completes the proof of the lemma.
\end{proof}

\begin{Corollary}\label{corcaccpp}
	Suppose that $qs_2\leq ps_1$ and $R\in (0,1)$. Let $u$ be a local weak solution to 
	\begin{align*}
		\pa_t u+\mc L u=f \quad\mbox{in }B_2\times (-2R^{ps_1},0]
	\end{align*}
	with $f\in L^\infty(B_1\times (-R^{ps_1},0])$ such that $u\in L^\infty(\mb R^N\times (-R^{ps_1},0])$. Then,
	\begin{align*}
		&\Big(R^{-N} \int_{-\frac{7}{8}R^{s_1p}}^{0} [u]_{W^{s_1,p}(B_R)}^p dt \Big)^\frac{1}{q} \nonumber\\
		&\leq C\big(\|u\|_{L^\infty(\mb R^N\times[-R^{s_1p},0])}+ 1+\|f\|_{L^\infty(B_1\times[-R^{s_1p},0])}\big),
	\end{align*}
	where $C=C(N,s_1,s_2,p,\|a\|_\infty)>0$ is a constant.
\end{Corollary}
\begin{proof}
	Set 
	\begin{align*}
		K_0:=\|u\|_{L^\infty(\mb R^N\times[-R^{s_1p},0])}+ 1+\|f\|_{L^\infty(B_1\times[-R^{s_1p},0])} \quad\mbox{and }\tl u=u+K_0.
	\end{align*}
	Then, $\tl u$ is still a local weak solution and satisfies $\tl u\geq 1$ in $\mb R^N\times [-R^{ps_1},0]$. Now, applying Lemma \ref{cacciopp} for the choice  $t_2=-7R^{s_1p}/8$, $\tau=R^{s_1p}/8$, $t_3=0$ together with $k=0$ and the test functions $\psi\in C^\infty_c(B_{R})$, $\eta\in C_c^\infty(\mb R)$ satisfying 
	\begin{align}\label{eq31}
		\psi\equiv 1 \mbox{ in }B_{3R/4}, \ 0\leq \psi\leq 1, 	\ |\na\psi|\leq CR^{-1};\nonumber\\
		\eta\equiv 1 \mbox{ in }[-7R^{ps_1}/8,0], \ 0\leq \eta\leq 1, \ |\eta^\prime|\leq CR^{-ps_1},
	\end{align}
	we obtain
	\begin{align*}
		&\int_{-\frac{7}{8}R^{s_1p}}^{0} [\tl u]_{W^{s_1,p}(B_{3R/4})}^p dt \nonumber\\
		&\leq C \sum_{(\el,s)} \int_{-R^{s_1p}}^{0}\int_{B_R}\int_{B_R} a_\el(x,y) \frac{|\psi(x)-\psi(y)|^\el}{|x-y|^{N+s\el}}(\tl u(x,t)+\tl u(y,t))^\el dxdydt \nonumber\\
		&\ +C  \sum_{(\el,s)} \sup_{\substack{-R^{s_1p}<t<0 \\ x\in supp(\psi)}} \Bigg( \int_{\mb R^N\setminus B_R}a_\el(x,y) \frac{u(y,t)_+^{\el-1}dy}{|x-y|^{N+s\el}}\Bigg) \\ &\qquad\quad\times\int_{-R^{s_1p}}^{0}\int_{B_R}\tl u(x,t)\psi^\el(x)dxdt \nonumber\\
		& \ + C \int_{-R^{s_1p}}^{0}\int_{B_R} |f(x,t)| \tl u(x,t)\psi^q(x)dxdt \nonumber\\
		& \ +C \int_{-R^{s_1p}}^{0} \int_{B_R}\tl u^2(x,t)\psi^q(x)|\pa_t\eta(t)|dxdt.
	\end{align*}
	Using \eqref{eq31}, it is easy to deduce that
	\begin{align*}
		&\int_{-\frac{7}{8}R^{s_1p}}^{0} [\tl u]_{W^{s_1,p}(B_{3R/4})}^p dt \nonumber\\
		&\leq C\|a\|_\infty\sum_{(\el,s)} \|u\|_{L^\infty(\mb R^N\times[-R^{s_1p},0])}^\el R^N R^{ps_1-\el s} \nonumber\\
		&\quad+CR^{N+ps_1}\|f\|_{L^\infty(B_1\times[-R^{s_1p},0])}+ C R^N \|u\|_{L^\infty(\mb R^N\times[-R^{s_1p},0])}^2 \nonumber\\
		&\leq CK_0^q R^N,
	\end{align*}
	where we have used $qs_2\leq ps_1$. This proves the corollary. 
\end{proof}
\subsection{The case of stationary problems}
We mention some preliminary results for the stationary problem.
We start with the following Caccioppoli type inequality without any non-negativity assumption on the solution. The proof runs along the similar lines of \cite[Lemma 3.2]{JDSacv} (see also \cite[Lemma 4.2]{byun}).
\begin{Lemma}\label{caccpstn}
	Let $u$ be a local weak solution to problem \eqref{probSt} such that either $u\in L^\infty_{{\rm loc}}(\Om)$ or $q\leq p^*_{s_1}$. Then, for $B_r\equiv B_r(x_0)\Subset\Om$ with $r>0$ and for all $\psi\in C_c^\infty(B_r)$ non-negative, we have
	\begin{align*}
		&\int_{B_r} \int_{B_r} \frac{|u(x)\psi^\frac{q}{p}(x)-u(y)\psi^\frac{q}{p}(y)|^p}{|x-y|^{N+s_1p}}dxdy \nonumber\\ 
		&\leq C  \int_{B_r} |f(x)| |u(x)|\psi^q(x)dx \nonumber\\
		&\ +C \sum_{(\el,s)} \int_{B_r}\int_{B_r} a_\el(x,y) \frac{|\psi(x)-\psi(y)|^\el}{|x-y|^{N+s\el}}(|u(x)|+|u(y)|)^\el dxdy \nonumber\\
		&\ + C   \sum_{(\el,s)} \int_{B_r}\int_{\mb R^N\setminus B_r}a_\el(x,y) \frac{|u(y)|^{\el-1}|u(x)|+|u(y)||u(x)|^{\el-1}}{|x-y|^{N+s\el}}\psi^q(x)dydx.
	\end{align*}
\end{Lemma}
Subsequently, following the proof of \cite[Proposition 3.1]{JDSjga}, we have the following local boundedness result (see also \cite[Theorem 1.1]{byun}).
\begin{Lemma}\label{localbdd}
	Suppose that  $q\leq p^*_{s_1}$. Let $u$ be a local weak solution to problem \eqref{probSt} in $\Om$.  Then, $u\in L^\infty_{{\rm loc}}(\Om)$.\\
	Additionally, suppose that  $q< p^*_{s_1}$. Then, for  $B_r\equiv B_r(x_0)\Subset\Om$ with $r\in (0,1)$, there exists a constant $C=C(N,s_1,p,s_2,q,\Bbbk,\|a\|_\infty)>0$, as a non-decreasing function of $\|a\|_\infty$,  such that
	\begin{align*}
		&\sup_{B_{3r/4}} |u| \nonumber\\
		&\leq  C \Big[\Big(\Xint-_{B_{r}}|u(x)|^{\vartheta}dx\Big)^\frac{q}{p\vartheta}+\sum_{(\el,s)} T^\el_{\el s,a_\el}(u;x_0,\frac{3r}{4})^\frac{1}{\el-1}
		+    \|f\|_{L^\Bbbk(B_r)}^\frac{1}{p-1} + 1 \Big],
	\end{align*}	
	where $\vartheta:=\max\{q,p\sg\}$ with  $\sg=\frac{p\Bbbk-(p^*_{s_1})'}{p(\Bbbk-(p^*_{s_1})')}\Big(<\frac{p^*_{s_1}}{p}\Big)$ if $\Bbbk<\infty$, while $\sg=1$ if $\Bbbk=\infty$.
\end{Lemma}
Now, we mention the notion of a weak solution to the exterior data Dirichlet problems and related energy spaces. 
Let $E$ and $E'$ are bounded domains such that $E\Subset E'\subset\mb R^N$. For given $g\in L^{p-1}_{ps_1}(\mb R^N)\cap L^{q-1}_{qs_2,a}(\mb R^N)$, we define
\begin{align*}%\label{bdry space}
	&X_{g,a}(E,E')\\
	&:= \{ v\in \mc W_a(E')\cap L^{p-1}_{ps_1}(\mb R^N)\cap L^{q-1}_{qs_2,a}(\mb R^N) : v=g \mbox{ a.e. in }\mb R^N\setminus E \},
\end{align*}
equipped with the norm of $\mc W_a(E')$. For  $f\in L^{\Bbbk}(E')$ with $\Bbbk>\max\{\frac{N}{ps_1},1\}$ and $g\in \mc W_a(E') \cap L^{p-1}_{s_1p}(\mb R^N) \cap L^{q-1}_{s_2q,a}(\mb R^N)$, consider the problem:
\begin{align}\label{prob extr data}
	\mc Lv=f \quad\mbox{in }E; \quad v=g \quad\mbox{in }\mb R^N\setminus E.
\end{align} 
\begin{Definition}%\label{defn bdry soln}
	A function $v\in X_{g,a}(E,E')$ is said to be a weak solution of problem \eqref{prob extr data} 	if, for all $\phi\in X_{0,a}(E,E')$, 
	\begin{align*}%\label{solndef}
		&\iint_{\mb R^{2N}} \Bigg(\frac{[v(x)-v(y)]^{p-1}}{|x-y|^{N+s_1p}}+a(x,y) \frac{[v(x)-v(y)]^{q-1}}{|x-y|^{N+s_2q}}\Bigg)(\phi(x)-\phi(y))dxdy \nonumber\\
		&= \int_{E}f(x)\phi(x)dx.
	\end{align*}
\end{Definition}
The existence of a weak solution $v\in X_{g,a}(E,E')$ to problem \eqref{prob extr data} is guaranteed by an analogous result of Theorem \ref{thm exst Appn}.
Again, in the sequel, we will drop the term $a$ from the definition of the space $X_{g,a}$, if the context is clear.
Concerning the boundedness properties of a weak solution $v$ to problem \eqref{prob extr data}, we see that since it also satisfies the same Caccioppoli inequality of Lemma \ref{caccpstn}, the local boundedness result of Lemma \ref{localbdd} holds for $v$, too. 
For the boundary estimates, following the idea of the proof of \cite[Theorem 5]{korvenpaa} and Lemma \ref{localbdd} (note that $v\in\mc W(E')$), we have 
\begin{Lemma}%\label{bdd bdry}
	Suppose that $q\leq p^*_{s_1}$.  Let $v\in X_{g}(E,E')$ be a weak solution to \eqref{prob extr data}. Additionally, assume that for $x_0\in \partial E$ and $0<r<\min\{1,{\rm dist}(x_0,E')\}$, $f\in L^{\Bbbk}(B_{r}(x_{0}))$ and $g\in L^{\infty}(B_{r}(x_{0}))$. Then, $v\in L^\infty (B_{r/4}(x_0))$. \\
	Furthermore, for the case $q< p^*_{s_1}$, there exists a constant $C>0$ depending only on $N,s_1,s_2,p,q,\Bbbk$ and $\|a\|_\infty$ such that 
	\begin{equation*}
		\begin{aligned}
			&\|v\|_{L^{\infty}(B_{r/4}(x_{0}))} \\
			&\leq  C \Big[\Big(\Xint-_{B_{r}(x_0)}|u(x)|^{\vartheta}dx\Big)^\frac{q}{p\vartheta}+\sum_{(\el,s)} T^\el_{\el s,a_\el}(u;x_0,\frac{r}{4})^\frac{1}{\el-1}
			+    \|f\|_{L^\Bbbk(B_r(x_0))}^\frac{1}{p-1}+1  \Big],
		\end{aligned}
	\end{equation*}
	where $\vartheta$ is given by Lemma \ref{localbdd}.
\end{Lemma}
Next, we mention the following H\"older continuity result with unspecified H\"older exponent for local weak solutions to problem \eqref{probSt}.
\begin{Proposition}\label{holdunspstn}
	Suppose that $p\geq 2$ and $qs_2\leq ps_1$. Let $u$ be a locally bounded local weak solution to problem \eqref{probSt} such that $u\in L^{q-1}_{qs_2}(\mb R^N)$. Then, there exists $\al_1>0$ such that $u\in C^{0,\al_1}_{{\rm loc}}(\Om)$. More precisely, if $B_{2R}\equiv B_{2R}(x_0)\Subset\Om$, for some $R\in (0,1)$, then for all $r\in (0,R/2)$, 
	\begin{align*}
		&[u]_{C^{0,\al_1}(B_{r/2})} \nonumber\\
		&\leq \frac{C}{R^{\al_1}} \bigg[ \|u\|_{L^\infty(B_r)}+ \sum_{(\el,s)}T^\el_{\el s,a_\el}(u;x_0,R/2)^\frac{1}{\el-1}
		+    \|f\|_{L^\Bbbk(B_r)}^\frac{1}{p-1}+1
		%\sum_{(\el,s)}T^\el_{\el s,a_\el}(u;x_0,R/2)^\frac{1}{\el-1}+ \|f\|_{L^\Bbbk(B_{R/2})}^\frac{1}{p-1} +1
		\bigg],
	\end{align*}
	where  $C$ is a positive constant depending only on $N,s_1,s_2,p,q,\Bbbk$ and $\|a\|_\infty$.
\end{Proposition}
\begin{proof}
	Let $v\in X_u(B_{3R/2},B_{7R/4})\cap L^{q-1}_{qs_2}(\mb R^N)$ be the unique solution to the problem:
	\begin{equation*}
		\left\{
		\begin{array}{rllll}
			\mc Lv&=0 \quad\mbox{in }B_{3R/2},\\
			v&=u \quad\mbox{in }\mb R^N\setminus B_{3R/2}.
		\end{array}
		\right.
	\end{equation*}
	Then, by a careful inspection of the proof \cite[Lemma 3.5]{fang}, specifically the definition of $\omega(r_0)$, there (see also the proof of \cite[Lemma 5.3]{byun}), for all $\rho<3R/4$, we have
	\begin{align*}
		\Xint-_{B_{\rho}} |v(x)-(v)_{x_0,\rho}|^p dx\leq C \Big(\frac{\rho}{R}\Big)^{p\al} \bigg[ \|v\|_{L^\infty(B_{2\rho})}^{p} +\sum_{(\el,s)} T^\el_{\el s,a_\el}(v;x_0,\rho)^\frac{p}{\el-1} +1\bigg],
	\end{align*}
	where we have further used the relation
	\begin{align*}
		\Xint-_{B_{\rho}} |v(x)-(v)_{x_0,\rho}|^p dx\leq C \Xint-_{B_\rho} ({\rm osc}_{B_\rho}v)^p dx.
	\end{align*}
	Moreover, taking $u-v$ as a test function and applying H\"older's inequality, we can prove that
	\begin{align*}
		[u-v]_{W^{s_1,p}(\mb R^N)}^p\leq C R^{N\big(\frac{p}{\Bbbk(p-1)}-\frac{N-ps_1}{N(p-1)}\big)} \|f\|_{L^\Bbbk(B_{R/2})}^\frac{1}{p-1}.
	\end{align*}
	Consequently, using the local boundedness result of Lemma \ref{localbdd}, the tail terms can be estimated as
	\begin{align*}
		T^\el_{\el s,a_\el}(v;x_0,R/2)^\frac{p}{\el-1}&\leq C T^\el_{\el s,a_\el}(u;x_0,3R/2)^\frac{p}{\el-1} \nonumber\\
		&\quad+C\Big(\Xint-_{B_{3R/2}} |u|^\el dx + \Xint-_{B_{3R/2}} |u-v|^\el dx\Big)^{p/\el}.
	\end{align*}
	Finally, using the relation
	\begin{align*}
		\Xint-_{B_r} |u(x)-(u)_{x_0,r}|^p dx\leq C \Xint-_{B_r} |u-v|^p dx+ C \Xint-_{B_r} |v(x)-(v)_{x_0,r}|^p dx,
	\end{align*}
	we get the required result of the proposition (see \cite[Theorem 3.6]{brascoH} for the further details). 
\end{proof}

\section{Higher interior regularity}	
In this section, we obtain higher space-time H\"older continuity result and prove Theorem \ref{thminh}.
\subsection{Spatial regularity: locally translation invariant case}
%In this section we prove our higher spatial regularity. 
We first fix some notations which will be used in this section.
For a measurable function $u:\mb R^N\to \mb R$ and $h\in\mb R^N$, we define 
\begin{align*}
	&u_h(x)=u(x+h), \ \  \de_h u(x)=u_h(x)-u(x)\quad\mbox{and}\\  &\de^2_h u(x)= \de_h(\de_h u(x))= u_{2h}(x)+u(x)-2u_h(x).
\end{align*} 
Now we present a regularization lemma involving the discrete differentials.
\begin{Lemma}\label{auxtest}
	Let $u$ be a local weak solution to problem \eqref{probM} in $B_{2}\times (-2,0]$. Suppose that $u\in L^\infty(E\times [-1,0])$ for all $E\Subset B_2$. Let $\eta\in C^{0,1}_c(B_2)$ be a non-negative function and let $\tau\in C^\infty(\mb R)$ be such that $0\leq \tau\leq 1$, $\tau\equiv 0$ in $(-\infty,T_0]$ and $\tau\equiv 1$ in $[T_1,\infty)$, for some $-1<T_0<T_1<0$. Then, for any locally Lipschitz function $\Psi:\mb R\to \mb R$ and $h\in\mb R^N$ such that $0<|h|<{\rm dist}({\rm supp}\eta,\pa B_2)/4$, we have 
	\begin{align}\label{eqi1}
		&\int_{T_0}^{T_1}\iint_{\mb R^{2N}}  \big([u_h(x,t)-u_h(y,t)]^{p-1}-[u(x,t)-u(y,t)]^{p-1}\big)\nonumber\\
		&\qquad\qquad\qquad\times(\Psi(\de_hu(x,t))\eta^p(x)-\Psi(\de_hu(y,t))\eta^p(y))\tau(t)d\mu_1 dt \nonumber\\ 
		&+\int_{T_0}^{T_1}\iint_{\mb R^{2N}}  \big(a_h(x,y)[u_h(x,t)-u_h(y,t)]^{q-1}-a(x,y)[u(x,t)-u(y,t)]^{p-1}\big)\nonumber\\
		&\qquad\qquad\qquad\times(\Psi(\de_hu(x,t))\eta^p(x)-\Psi(\de_hu(y,t))\eta^p(y))\tau(t)d\mu_2dt\nonumber\\
		&+\int_{B_2}\widetilde\Psi(\de_hu(x,T_1))\eta^p(x) dxdt \nonumber\\
		&=\int_{T_0}^{T_1}\int_{B_2}\widetilde\Psi(\de_hu(x,t))\eta^p(x)\tau^\prime(t) dxdt \nonumber\\ &\quad+\int_{T_0}^{T_1}\int_{B_2}\big(f_h-f\big)(x,t)\Psi(u_h(x,t))dxdt,
	\end{align}
	where $\widetilde\Psi(z)=\int_{0}^{z}\Psi(\sg)d\sg$, $a_h(x,y)=a(x+h,y+h)$ and $f_h(x,t)=f(x+h,t)$.
\end{Lemma}
\begin{proof}
	We only sketch the proof as it is similar to the one provided in \cite[Lemma 3.3]{brascoP}. Set $J:=[T_0,T_1]\subset (-1,0)$ and let $\phi\in \mc Y(J;\mc W(B_2))\cap C^1(J; L^2(B_2))$ be such that its spatial support is compactly contained in $B_2$. Let $\phi^\e$ be the regularization of $\phi$, as defined in Definition \ref{regtst}, for
	\begin{align*}
		0<\e<\frac{\e_0}{2}:=\frac{1}{4}\min\{T_0+1,-T_1,T_1-T_0\}.
	\end{align*}
	Then, proceeding similar to the proof of \cite[Lemma 3.3]{brascoP}, we have 
	\begin{align}\label{eqi40}
		&\Upsilon_u(\e)+\int_{T_0+\frac{\e}{2}}^{T_1-\frac{\e}{2}} \int_{B_2}\pa_t u^\e(x,t)\phi(x,t)dxdt \nonumber\\
		&+ \int_{J} \iint_{\mb R^{2N}} \frac{[u(x,t)-u(y,t)]^{p-1}}{|x-y|^{N+s_1p}}(\phi^\e(x,t)-\phi^\e(y,t))dxdydt \nonumber\\
		&+\int_{J} \iint_{\mb R^{2N}}a(x,y) \frac{[u(x,t)-u(y,t)]^{q-1}}{|x-y|^{N+s_2q}}(\phi^\e(x,t)-\phi^\e(y,t))dxdydt \nonumber \\
		&= \int_{J} \int_{B_2}f(x,t)\phi^\e(x,t)dx dt +\int_{B_2} u(x,T_0) \phi^\e(x,T_0)dx \nonumber\\
		&- \int_{B_2} u^\e(x,T_0+\frac{\e}{2}) \phi(x,T_0+\frac{\e}{2})dx \nonumber\\
		& \ +\int_{B_2} u^\e(x,T_1-\frac{\e}{2}) \phi(x,T_1-\frac{\e}{2})dx - \int_{B_2} u(x,T_1) \phi^\e(x,T_1)dx,
	\end{align}
	where 
	\begin{align}\label{eqi41}
		\Upsilon_u(\e):=&-\int_{B_2} \Bigg(\frac{1}{\e}\int_{T_0}^{T_0+\e}u(x,t)\varsigma\Big(\frac{T_0-t}{\e}+\frac{1}{2}\Big)dt \Bigg)\phi(x,T_0+\frac{\e}{2}) dx \nonumber\\
		&+\int_{B_2} \int_{-\frac{1}{2}}^{\frac{1}{2}}\Bigg(\int_{-\frac{1}{2}}^{\rho}u(x,\e\rho+T_0-\e\sg)\varsigma^\prime(\sg)d\sg \Bigg)\phi(x,\e\rho+T_0)d\rho dx \nonumber\\
		&+\int_{B_2} \Bigg(\frac{1}{\e}\int_{T_1-\e}^{T_1}u(x,t)\varsigma\Big(\frac{T_1-t}{\e}-\frac{1}{2}\Big)dt \Bigg)\phi(x,T_1-\frac{\e}{2}) dx \nonumber\\
		&-\int_{B_2} \int_{-\frac{1}{2}}^{\frac{1}{2}}\Bigg(\int_{\frac{1}{2}}^{\rho}u(x,\e\rho+T_1-\e\sg)\varsigma^\prime(\sg)d\sg \Bigg)\phi(x,\e\rho+T_1)d\rho dx.
	\end{align}
	Again testing the weak formulation with $\phi_{-h}(x,t):=\phi(x-h,t)$ and changing the variables, we see that \eqref{eqi40} and \eqref{eqi41} hold for $u_h$ with $a(\cdot,\cdot)$ and $f(x,t)$ replaced by $a_h(x,y)$ and  $f_h(x,t)$, respectively. On subtracting the resulting equations and taking 
	\begin{align*}
		\phi=\Psi(u_h^\e-u^\e)\eta^p\tau_\e=\Psi(\de_h u^\e)\eta^p\tau_\e,
	\end{align*}
	where 
	\begin{align*}
		\tau_\e(t)=\tau\Bigg(\frac{T_1-T_0}{T_1-T_0-\e}\Big(t-T_1+\frac{\e}{2}\Big)+T_1\Bigg),
	\end{align*}
	we obtain
	\begin{align}\label{eqi42}
		&\int_{T_0}^{T_1}\iint_{\mb R^{2N}}  \big([u_h(x,t)-u_h(y,t)]^{p-1}-[u(x,t)-u(y,t)]^{p-1}\big) \nonumber\\
		&\qquad\qquad\times\Big(\big(\Psi(\de_hu^\e(x,t))\tau_\e(t)\big)^\e\eta^p(x)-\big(\Psi(\de_hu^\e(y,t))\tau_\e(t)\big)^\e\eta^p(y)\Big) d\mu_1 dt \nonumber\\ 
		&+\int_{T_0}^{T_1}\iint_{\mb R^{2N}}  \big(a_h(x,y)[u_h(x,t)-u_h(y,t)]^{q-1}-a(x,y)[u(x,t)-u(y,t)]^{p-1}\big)\nonumber\\
		&\qquad\qquad\times\Big(\big(\Psi(\de_hu^\e(x,t))\tau_\e(t)\big)^\e\eta^p(x)-\big(\Psi(\de_hu^\e(y,t))\tau_\e(t)\big)^\e\eta^p(y)\Big)d\mu_2dt\nonumber\\
		&+\int_{T_0}^{T_1} \int_{B_2}\pa_t (\de_hu^\e(x,t))\Psi(\de_hu^\e(x,t))\eta^p(x)\tau_\e(t)dxdt+ \big(\Upsilon_{u_h}(\e)-\Upsilon_u(\e)\big) \nonumber\\
		&=  \int_{T_0}^{T_1} \int_{B_2}(f_h-f)(x,t)\big(\Psi(\de_hu^\e(x,t))\tau_\e(t)\big)^\e\eta^p(x)dx dt \nonumber\\
		& \ +\int_{B_2} \de_hu^\e(x,T_1-\frac{\e}{2}) \Psi(\de_hu^\e(x,T_1-\frac{\e}{2})) \eta^p(x)dx\nonumber\\
		& \ -\int_{B_2} \de_hu(x,T_1) \big(\Psi(\de_hu^\e(x,T_1))\big)^\e\eta^p(x)dx.
	\end{align}
	Using the integration by parts formula, the definition of $\widetilde{\Psi}$ and convergence arguments similar to the proof of \cite[Lemma 3.3]{brascoP}, we get, as $\e\to 0$,
	\begin{align*}
		&\int_{T_0}^{T_1} \int_{B_2}\pa_t (\de_hu^\e)\Psi(\de_hu^\e)\eta^p(x)\tau_\e(t)dxdt\nonumber\\
		&=\int_{B_2}\widetilde\Psi(\de_hu^\e(x,T_1))\eta^p(x)dx-\int_{T_0}^{T_1} \int_{B_2} \widetilde\Psi(\de_hu^\e)\eta^p(x)\tau_\e^\prime(t)dxdt	\nonumber\\
		&\to \int_{B_2}\widetilde\Psi(\de_hu(x,T_1))\eta^p dxdt -\int_{T_0}^{T_1}\int_{B_2}\widetilde\Psi(\de_hu)\eta^p\tau^\prime dxdt.
	\end{align*}
	An analogous treatment implies that, as $\e\to 0$, 
	\begin{align*}
		&\int_{T_0}^{T_1} \int_{B_2}(f_h-f)\big(\Psi(\de_hu^\e(x,t))\tau_\e(t)\big)^\e\eta^p(x)dx dt \nonumber\\
		&\to \int_{T_0}^{T_1} \int_{B_2}(f_h-f)\Psi(\de_hu(x,t))\tau(t)\eta^p(x)dx dt. 
	\end{align*}
	Moreover, similar convergence arguments as in \cite[Lemma 3.3]{brascoP} yield that the second term on the right hand side of \eqref{eqi42} converges to $0$ along with
	\begin{align*}
		\Upsilon_{u_h}(\e)-\Upsilon_u(\e)\to 0, \quad\mbox{as }\e\to 0.
	\end{align*}
	Now, to check the convergence of the first two terms on the left side of \eqref{eqi42},  we observe that 
	\begin{align*}
		&\int_{T_0}^{T_1}\iint_{\mb R^{2N}} \big([u_h(x,t)-u_h(y,t)]^{p-1}-[u(x,t)-u(y,t)]^{p-1}\big) \nonumber\\
		&\qquad\quad\quad\times\Big(\big(\Psi(\de_hu^\e(x,t))\tau_\e(t)\big)^\e\eta^p(x)-\big(\Psi(\de_hu^\e(y,t))\tau_\e(t)\big)^\e\eta^p(y)\Big) d\mu_1 dt \nonumber\\ 
		&=\int_{T_0}^{T_1}\iint_{B_{2-2h}\times B_{2-2h}} \big([u_h(x,t)-u_h(y,t)]^{p-1}-[u(x,t)-u(y,t)]^{p-1}\big) \nonumber\\
		&\qquad\quad\quad\times\Big(\big(\Psi(\de_hu^\e(x,t))\tau_\e(t)\big)^\e\eta^p(x)-\big(\Psi(\de_hu^\e(y,t))\tau_\e(t)\big)^\e\eta^p(y)\Big) d\mu_1 dt \nonumber\\ 
		&+ 2\int_{T_0}^{T_1}\iint_{B_{2-2h}\times (\mb R^N\times B_{2-h})} \big([u_h(x,t)-u_h(y,t)]^{p-1}-[u(x,t)-u(y,t)]^{p-1}\big) \nonumber\\
		&\qquad\quad\quad\times\big(\Psi(\de_hu^\e(x,t))\tau_\e(t)\big)^\e\eta^p(x) d\mu_1 dt.
	\end{align*}
	Analogously the second integral on the left side of \eqref{eqi42} is split into two parts. As in \cite[Appendix B]{brascoP}, we obtain that
	the sequence
	$\{ \big(\Psi(\de_hu^\e(\cdot,t))\tau_\e(t)\big)^\e\eta^p \}$ is bounded in $\mc Y(J;\mc W(B_{2-2h}))$ with respect to $\e$. Then, proceeding similar to the proof of convergence of $\mathfrak{I}_2^\e(\el)$ and $\mf I_3^\e(\el)$ in Lemma \ref{cacciopp}, we get the required convergence. 
\end{proof}

\begin{Proposition}\label{propin1}
	Suppose that $p\geq 2$, $qs_2< ps_1+2(1-s_1)$ and the assumption {\rm\textbf{(A.1)}} holds. Let  $u$ be a local weak solution to the following problem:
	\begin{align*}
		\pa_t u+(-\De)_p^{s_1} u+\la(-\De)_{q,a}^{s_2} u =f \quad\mbox{in }B_{2}\times (-2,0],
	\end{align*}
	where $\la>0$ is a parameter and $f\in L^\infty(B_1\times (-1,0])$, such that 
	\begin{align}\label{eqibdu}
		\| u\|_{L^\infty(B_1\times [-1,0])}\leq 1 \quad\mbox{and}\quad \sum_{(\el,s)} T^\el_{\infty,\el s}(u;0,1,-1,0)\leq 1.
	\end{align} 
	Additionally, suppose that  for some $m\geq p$ and $0<h_0<\frac{1}{10}$, we have 
	\begin{align*}
		\int_{T_0}^{T_1}\sup_{0<|h|<h_0} \bigg\| \frac{\de^2_h u(\cdot,t)}{|h|^{s_1}}\bigg\|^m_{L^m(B_{R+4h_0})} dt<\infty,
	\end{align*}
	for $4h_0<R\leq 1 - 5h_0$ and $-1<T_0<T_1\leq 0$. Then, for all $\mu\in (0,T_1-T_0)$, there holds: 
	\begin{align*}
		&\int_{T_0+\mu}^{T_1}\sup_{0<|h|<h_0} \bigg\| \frac{\de^2_h u(\cdot,t)}{|h|^{s_1}}\bigg\|^{m+1}_{L^{m+1}(B_{R-4h_0})} dt \nonumber\\
		&\quad+ \frac{1}{m+3-p} \sup_{0<|h|<h_0} \bigg\| \frac{\de_h u(\cdot,T_1)}{|h|^{\frac{(m+2-p)s_1}{m+3-p}}}\bigg\|^{m+3-p}_{L^{m+3-p}(B_{R-4h_0})} \\
		&\leq C(1+\la+\|f\|_{L^\infty(Q_1)}) \int_{T_0}^{T_1} \left( \sup_{0<|h|<h_0} \bigg\| \frac{\de^2_h u(\cdot,t)}{|h|^{s_1}}\bigg\|^m_{L^m(B_{R+4h_0})} +1 \right)dt,
	\end{align*}
	where $C= C(N,h_0,p,q,m,s_1,\|a\|_\infty)>0$ (which depends inversely on $h_0$). 
\end{Proposition}
\begin{proof}
	We first assume that $T_1<0$ and set $r:=R-4h_0$. 
	For $\ka\geq 1$, we take 
	\begin{align*}
		\Psi(z):=[z]^\ka 
	\end{align*}
	in Lemma \ref{auxtest}, together with
	\begin{align*}
		&\eta\equiv 1 \quad\mbox{in }B_r, \ \eta\equiv 0  \quad\mbox{in }B^c_{(R+r)/2}, \ \ 0\leq \eta \leq 1 \quad\mbox{and } |\na\eta|\leq \frac{C}{R-r}=\frac{C}{4h_0}; \\
		&0\leq \tau \leq 1, \quad \tau\equiv 0 \ \mbox{in }(-\infty,T_0], \quad \tau\equiv1 \ \mbox{in }[T_0+\mu,\infty) \quad\mbox{and } |\tau^\prime|\leq C\mu^{-1}.
	\end{align*}
	Dividing \eqref{eqi1} by $|h|^{1+\nu\ka}>0$,  we have 
	\begin{align*}
		&\int_{T_0}^{T_1} \iint_{\mb R^{2N}}  \frac{([u_h(x,t)-u_h(y,t)]^{p-1}-[u(x,t)-u(y,t)]^{p-1})}{|h|^{1+\nu\ka}} \nonumber \\
		&\qquad\qquad\qquad\times\big([\de_hu(x,t)]^\ka \eta^p(x)-[\de_hu(y,t)]^\ka \eta^p(y) \big) \tau(t)d\mu_1 dt\nonumber\\ 
		&+\int_{T_0}^{T_1}\iint_{\mb R^{2N}}  \frac{(a_h(x,y)[u_h(x,t)-u_h(y,t)]^{q-1}-a(x,y)[u(x,t)-u(y,t)]^{q-1})}{|h|^{1+\nu\ka}} \nonumber \\
		&\qquad\qquad\qquad\times \la\big([\de_hu(x,t)]^\ka \eta^p(x)-[\de_hu(y,t)]^\ka \eta^p(y) \big) \tau(t)d\mu_2 dt \nonumber \\
		&+\frac{1}{\ka+1} \int_{B_2} \frac{|\de_h u(x,T_1)|^{\ka+1}}{|h|^{1+\nu\ka}}\eta^p(x) dx  \nonumber\\
		&= \frac{1}{\ka+1} \int_{T_0}^{T_1} \int_{B_2} \frac{|\de_h u(x)|^{\ka+1}}{|h|^{1+\nu\ka}}\eta^p(x) \tau^\prime(t) dx dt \nonumber\\
		& \ 
		+\int_{T_0}^{T_1} \int_{B_2} \de_hf(x,t)\frac{|\de_h u(x)|^{\ka}}{|h|^{1+\nu\ka}}\eta^p(x) \tau(t) dx dt,
	\end{align*} 
	where $\de_hf(x,t)=f(x+h,t)-f(x,t)$. %For $(\el,s)\in \{(p,s_1),(q,s_2)\}$,  
	We set the following:
	\begin{align*}
		&\mb I_{1,q}(t):=\la\int_{B_R}\int_{B_R} a(x,y) \frac{([u_h(x,t)-u_h(y,t)]^{q-1}-[u(x,t)-u(y,t)]^{q-1})}{|h|^{1+\nu\ka}} \nonumber\\
		&\qquad\qquad\qquad\times\big([\de_hu(x,t)]^\ka \eta^p(x)-[\de_hu(y,t)]^\ka \eta^p(y) \big) d\mu_2, \\
		&\mb I_{2,q}(t)\nonumber\\
		&:= \int_{B_\frac{R+r}{2}}\int_{B_R^c} \frac{(a_h(x,y)[u_h(x,t)-u_h(y,t)]^{q-1}-a(x,y)[u(x,t)-u(y,t)]^{q-1})}{|h|^{1+\nu\ka}} \nonumber\\
		&\qquad\qquad\qquad\qquad\qquad\times \la[\de_hu(x,t)]^\ka \eta^p(x) d\mu_2, \nonumber\\
		&\mb I_3:=\frac{1}{\ka+1} \int_{T_0}^{T_1} \int_{B_2} \frac{|\de_h u(x,t)|^{\ka+1}}{|h|^{1+\nu\ka}}\eta^p(x) \tau^\prime(t) dx dt, \\
		&\mb I_{f}:= \int_{T_0}^{T_1} \int_{B_2} |\de_hf(x,t)|\frac{|\de_h u(x)|^{\ka}}{|h|^{1+\nu\ka}}\eta^p(x) \tau(t) dx dt
	\end{align*}
	and analogously, $\mb I_{1,p}(t)$ and $\mb I_{2,p}(t)$ are defined by taking $\la=1$ and $a\equiv 1$. Therefore, noting $a(x+h,y+h)=a(x,y)$ in $B_R\times B_R$, we have 
	\begin{align}\label{eqi4}
		&\int_{T_0}^{T_1} \mb I_{1,p}(t)\tau(t) dt + \frac{1}{\ka+1} \int_{B_2} \frac{|\de_h u(x,T_1)|^{\ka+1}}{|h|^{1+\nu\ka}}\eta^p dx \nonumber\\
		&\leq 2\int_{T_0}^{T_1}|\mb I_{2,q}(t)|\tau(t)dt +2\int_{T_0}^{T_1}|\mb I_{2,p}(t)|\tau(t) dt\nonumber \\
		&\quad-\int_{T_0}^{T_1} \mb I_{1,q}(t)\tau(t) dt+ \mb I_3+\mb I_{f}.
	\end{align}
	For notational convenience,  we suppress the $t$-dependence from  $u$ and $u_h$. \\
	\textit{Estimate of $\mb I_{1,p}$}. Proceeding similarly to Step 1 of the proof of \cite[Proposition 4.1]{brascoH}, we obtain
	\begin{align}\label{eqi3}
		\mb I_{1,p}
		&\geq c \bigg[\frac{[\de_hu ]^\frac{\ka+p-1}{p}\eta}{|h|^\frac{1+\nu\ka}{p}}\bigg]^p_{W^{s_1,p}(B_R)} \nonumber\\
		&\ - C \int_{B_R}\int_{B_R} \frac{|\de_hu(x)|^{\ka+p-1}+|\de_hu(y)|^{\ka+p-1}}{|h|^{1+\nu\ka}} |\eta(x)-\eta(y)|^p d\mu_1 \nonumber\\
		&\ -C\int_{B_R} \int_{B_R} \left[ \frac{|\de_hu(x)|^{\ka+1}+|\de_hu(y)|^{\ka+1}}{|h|^{1+\nu\ka}}\big|\eta^\frac{p}{2}(x)-\eta^\frac{p}{2}(y)\big|^2
		 \right. \nonumber\\
		&\qquad\qquad\qquad\left. \times \big( |u_h(x)-u_h(y)|^\frac{p-2}{2}+|u(x)-u(y)|^\frac{p-2}{2} \big)^2 \right]d\mu_1 \nonumber \\
		&=:c \bigg[\frac{[\de_hu ]^\frac{\ka+p-1}{p}\eta}{|h|^\frac{1+\nu\ka}{p}}\bigg]^p_{W^{s_1,p}(B_R)} - \mb I_{4,p}- \mb I_{5,p}. 
	\end{align}
	Moreover, from the proof of \cite[Proposition 3.9]{JDSacv}, we have 
	\begin{align}\label{eqi2}
		&|\mb I_{4,p}|+|\mb I_{5,p}| \nonumber\\
		&\leq C\left(\sup_{0<|h|<h_0}\bigg\| \frac{\de^2_h u}{|h|^{s_1}}\bigg\|^m_{L^m(B_{R+4h_0})} +1 + \int_{B_R}  \frac{|\de_hu(x)|^{\frac{\ka m}{m-p+2}}}{|h|^\frac{(1+\nu\ka)m}{m-p+2}}dx \right).
	\end{align}
	\textit{Estimate of $\mb I_{1,q}$}. Following the proof of \cite[Equ. (3.16)]{JDSacv} (note that the first term is non-negative there), we obtain
	\begin{align}\label{eqi6}
		\mb I_{1,q} &\geq -C\la\int_{B_R}\int_{B_R} a(x,y)\left[ \frac{|\de_hu(x)|^{\ka+1}+|\de_hu(y)|^{\ka+1}}{|h|^{1+\nu\ka}} 
		 \big|\eta^\frac{p}{2}(x)-\eta^\frac{p}{2}(y)\big|^2 \right. \nonumber\\
		& \left.  \qquad\qquad\qquad\times \big( |u_h(x)-u_h(y)|^\frac{q-2}{2}+|u(x)-u(y)|^\frac{q-2}{2} \big)^2  \right]d\mu_2 \nonumber \\
		&=:- C \la\mb I_{4,q}. 
	\end{align} 
	In order to estimate $\mb I_{4,q}$, we set
	\begin{align*}
		\mb {\tl I}_{4,q}:&=\la\int_{B_R} \int_{B_R} a(x,y) \frac{|\de_hu(x)|^{\ka+1}}{|h|^{1+\nu\ka}} |u(x)-u(y)|^{q-2} \big|\eta^\frac{p}{2}(x)-\eta^\frac{p}{2}(y)\big|^2 d\mu_2.
	\end{align*}
	For the case $q=2$, using the bounds on $\eta$ and $\na\eta$, we get 
	\begin{align*}
		\mb {\tl I}_{4,q}&\leq \frac{C(N,s)}{h_0^2} \la\|a\|_{\infty} \|u\|_{L^\infty(B_{R+h_0})} \int_{B_R}\frac{|\de_h u(x)|^\ka}{|h|^{1+\nu\ka}}dx.
	\end{align*}  
	For the case $q>2$, observe that $qs_2\leq ps_1+\sg$, for $\sg<2(1-s_1)$.
	Therefore, employing Young's inequality and the bounds on $u$, $a(\cdot,\cdot)$ and $\eta$, we deduce that
	\begin{align*}
		\mb {\tl I}_{4,q} &\leq CR^{ps_1+\sg-qs_2}\int_{B_R}\int_{B_R} \la a(x,y)  \frac{|u(x)-u(y)|^{q-2}}{|x-y|^{N+s_1p+\sg}} \frac{|\de_hu(x)|^{\ka+1}}{|h|^{1+\nu\ka}} \nonumber\\
		&\qquad\qquad\qquad\times \big|\eta^\frac{p}{2}(x)-\eta^\frac{p}{2}(y)\big|^2 dxdy \\
		&\leq \frac{C\la  \|a\|_{\infty}\|u\|_{L^\infty(B_R)}^{q-p+1}}{h_0^2} \int_{B_R}\int_{B_R}   \frac{|u(x)-u(y)|^{p-2}}{|x-y|^{N+s_1p+\sg-2}} \frac{|\de_hu(x)|^{\ka}}{|h|^{1+\nu\ka}}  dxdy \\
		&\leq  C\int_{B_R}\int_{B_R}   \frac{|u(x)-u(y)|^{m}}{|x-y|^{N+ms_1\frac{(p-2-\e)}{p-2}}}\,dxdy\nonumber\\
		&\quad+ C \int_{B_R}\int_{B_R}   |x-y|^{\frac{m(2-2s_1-\e s_1-\sg)}{m-p+2}-N} \frac{|\de_hu(x)|^{\frac{\ka m}{m-p+2}}}{|h|^\frac{(1+\nu\ka)m}{m-p+2}}\,dxdy \nonumber\\
		&\leq C [u]^m_{W^{\frac{s_1(p-2-\e)}{p-2},m}(B_{R+h_0})} + C  \int_{B_R}  \frac{|\de_hu(x)|^{\frac{\ka m}{m-p+2}}}{|h|^\frac{(1+\nu\ka)m}{m-p+2}}dx, 
	\end{align*}
	where we have used the relation $2-2s_1-\e s_1-\sg>0$, that is, $\e\in (0,\frac{2-\sg}{s_1}-2)$ (thanks to $\sg<2(1-s_1)$) in the last inequality. Note that the right hand side of the above expression is exactly the same as the estimate of $\tl{I}_{11}$, given by (3.20) in \cite{JDSacv}. Thus, similar to \cite[Equ. (3.22)]{JDSacv}, we have
	{\small\begin{align*}%\label{eqi9}
			\mb {\tl I}_{4,q} \leq C \la \Bigg(\sup_{0<|h|<h_0}\bigg\| \frac{\de^2_h u}{|h|^{s_1}}\bigg\|^m_{L^m(B_{R+4h_0})} +1 + \int_{B_R}  \frac{|\de_hu(x)|^{\frac{\ka m}{m-p+2}}}{|h|^\frac{(1+\nu\ka)m}{m-p+2}}dx \Bigg).
	\end{align*}}
	Consequently,  from \eqref{eqi6}, we obtain
	\begin{align}\label{eqi5}
		-\mb I_{1,q}\leq c{\mb I}_{4,q} \leq c \la \Bigg[\sup_{0<|h|<h_0}\bigg\| \frac{\de^2_h u}{|h|^{s_1}}\bigg\|^m_{L^m(B_{R+4h_0})}  + \int_{B_R}  \frac{|\de_hu(x)|^{\frac{\ka m}{m-p+2}}}{|h|^\frac{(1+\nu\ka)m}{m-p+2}}dx+1 \Bigg].
	\end{align}
	\textit{Estimate of $\mb I_{2,\el}$}. Noting  $|u|\leq 1$ on $B_1\times[-1,0]$, we have 
	\begin{align*}
		\big|a_\el[u_h(x)-u_h(y)]^{\el-1}[\de_h u(x)]^{\ka}\big|\leq C\|a_\el\|_\infty (1+|u_h(y)|^{\el-1})|\de_h u(x)|^\ka.
	\end{align*}
	Therefore, on account of \eqref{eqibdu}, we deduce that
	\begin{align*}
		&\int_{B_\frac{R+r}{2}}\int_{B_R^c} \frac{a_\el(x+h,y+h)[u_h(x,t)-u_h(y,t)]^{\el-1}}{|h|^{1+\nu\ka}} [\de_hu(x)]^\ka \eta^p(x) d\mu \nonumber\\
		&\leq C  \int_{B_\frac{R+r}{2}}\frac{|\de_hu(x)|^\ka}{|h|^{1+\nu\ka}}dx \int_{B_R^c} \frac{1+|u_h(y)|^{\el-1}}{|y|^{N+s\el}}dy \nonumber\\
		&\leq  \frac{C}{(R-r)^{s\el}}\Bigg( 1 + \int_{B_R}  \frac{|\de_hu(x)|^{\frac{\ka m}{m-p+2}}}{|h|^\frac{(1+\nu\ka)m}{m-p+2}}dx \Bigg).
	\end{align*}
	Consequently, we have
	\begin{align}\label{eqi10}
		|\mb I_{2,p}|+|\mb I_{2,q}| \leq C(1+\la) \Bigg( 1 + \int_{B_R}  \frac{|\de_hu(x)|^{\frac{\ka m}{m-p+2}}}{|h|^\frac{(1+\nu\ka)m}{m-p+2}}dx \Bigg).
	\end{align}
	\textit{Estimate of $\mb I_3$  and $\mb I_f$}.
	Using \eqref{eqibdu} once again, we get
	\begin{align}\label{eqi11}
		\mb I_3 \leq \frac{C}{\mu} \int_{T_0}^{T_1} \Bigg( 1 + \int_{B_R}  \frac{|\de_hu(x)|^{\frac{\ka m}{m-p+2}}}{|h|^\frac{(1+\nu\ka)m}{m-p+2}}dx \Bigg)dt
	\end{align}
	and noting $f\in L^\infty(B_R\times[T_0,T_1])$, we infer that
	\begin{align}\label{eqif}
		\mb I_{f} &\leq 2\|f\|_{L^\infty(B_R\times[T_0,T_1])} \int_{T_0}^{T_1} \int_{B_R} \frac{|\de_h u(x)|^{\ka}}{|h|^{1+\nu\ka}} dx dt \nonumber\\
		&\leq C \int_{T_0}^{T_1}\Big( 1 + \int_{B_R}  \frac{|\de_hu(x)|^{\frac{\ka m}{m-p+2}}}{|h|^\frac{(1+\nu\ka)m}{m-p+2}}dx \Big)dt.
	\end{align}
	Combining \eqref{eqi3}, \eqref{eqi2}, \eqref{eqi5}, \eqref{eqi10}, \eqref{eqi11} and \eqref{eqif} with \eqref{eqi4}, we obtain
	\begin{align*}
		&\int_{T_0+\mu}^{T_1}\bigg[\frac{[\de_hu ]^\frac{\ka+p-1}{p}\eta}{|h|^\frac{1+\nu\ka}{p}}\bigg]^p_{W^{s_1,p}(B_R)} dt + \frac{1}{\ka+1} \int_{B_R} \frac{|\de_h u(x,T_1)|^{\ka+1}}{|h|^{1+\nu\ka}}\eta^p dx \nonumber\\
		&\leq C_{\la,f} \int_{T_0}^{T_1} \Bigg(\int_{B_R} \frac{|\de_hu(x)|^{\frac{\ka m}{m-p+2}}}{|h|^\frac{(1+\nu\ka)m}{m-p+2}}dx + \sup_{0<|h|<h_0} \bigg\| \frac{\de^2_h u}{|h|^{s_1}}\bigg\|^m_{L^m(B_{R+4h_0})}+1 \Bigg)dt,
	\end{align*}
	where $C_{\la,f}=C(N,p,s_1,h_0,q,m,\|a\|_\infty)(1+\la+\|f\|_{L^\infty(B_1\times(-1,0])})$ is a constant and we have used the fact that $\tau\equiv 1$ on $[T_0+\mu,T_1]$. Then, for the choice of $(1+\nu\ka)/\ka<1$, following the calculation to derive \cite[Equ. (3.35)]{JDSacv}, we get
	\begin{align*}%\label{eqi15}
		&\int_{T_0+\mu}^{T_1} \sup_{0<|h|<h_0} \int_{B_r}   \bigg|\frac{\de^2_hu}{|h|^\frac{1+\nu\ka+s_1p}{\ka-1+p}}\bigg|^{\ka-1+p} dxdt \nonumber\\
		&+ \frac{1}{\ka+1} \sup_{0<|h|<h_0}\int_{B_R} \frac{|\de_h u(x,T_1)|^{\ka+1}}{|h|^{1+\nu\ka}}\eta^p dx \nonumber\\
		&\leq C_{\la,f} \int_{T_0}^{T_1} \int_{B_R}\sup_{0<|h|<h_0}  \Bigg(\bigg|\frac{\de^2_hu}{|h|^\frac{1+\nu\ka}{\ka}}\bigg|^\frac{m\ka}{m-p+2} dx+  \bigg\| \frac{\de^2_h u}{|h|^{s_1}}\bigg\|^m_{L^m(B_{R+4h_0})}+1 \Bigg)dt.	
	\end{align*}
	The above expression is same as \cite[Equ. (4.11)]{brascoP}, thus, from step 6 of the same paper, we deduce that
	\begin{align*}
		&\int_{T_0+\mu}^{T_1} \sup_{0<|h|<h_0} \bigg\| \frac{\de^2_h u}{|h|^{s_1}}\bigg\|^{m+1}_{L^{m+1}(B_{R-4h_0})}dt \nonumber\\
		&\quad+ \frac{1}{m+3-p} \sup_{0<|h|<h_0}\int_{B_{R-4h_0}} \frac{|\de_h u(x,T_1)|^{m+3-p}}{|h|^{(m+2-p)s_1}} dx \nonumber\\
		&\leq C(1+\la+\|f\|_{L^\infty(B_1\times(-1,0])}) \int_{T_0}^{T_1} \Bigg(\sup_{0<|h|<h_0} \bigg\| \frac{\de^2_h u}{|h|^{s_1}}\bigg\|^m_{L^m(B_{R+4h_0})}+1 \Bigg)dt.
	\end{align*}
	For the case $T_1=0$, we use an approximation argument as in step 7 of  \cite[Proposition 4.1]{brascoP}. 
\end{proof}	

\begin{Theorem}\label{thmins}
	Suppose that the function $a(\cdot,\cdot)$ satisfies the assumption {\rm\textbf{(A.1)}} and $qs_2\leq ps_1$. Let $u$ be a local weak solution to problem \eqref{probM} such that  $u\in L^\infty_{{\rm loc}}(\Om\times I)\cap L^{\infty}_{\rm loc}(I;L^{p-1}_{ps_1}(\mb R^N)) \cap L^{q-1}_{\rm loc}(I; L^{q-1}_{qs_2}(\mb R^N))$. 
	Then, for every $\al\in (0,s_1)$, $u\in C^{0,\al}_{x, {\rm loc}}(\Om\times I)$. More precisely, for every $\al\in (0,s_1)$ and $R\in (0,1)$ satisfying $Q_{2R,2R^{s_1p}}\equiv Q_{2R,2R^{s_1p}}(x_0,0)\Subset\Om\times I$, there exists a constant $C=C(N,s_1,p,q,\al,\|a\|_\infty)>0$ such that 
	\begin{align*}
		\sup_{t\in[-\frac{R^{s_1p}}{2},0]} [u(\cdot,t)]_{C^{0,\al}(B_\frac{R}{2})} \leq \frac{C}{R^{\al}} \mc K_R^{(q-p)i_\infty+1},
	\end{align*}
	where $i_\infty\in\mb N$ is such that $i_\infty>\frac{3\al+(N-2)s_1}{s_1-\al}$ and
	\begin{align}\label{K R}
		\mc K_R:=  &\Big[\|u\|_{L^\infty(Q_{R,R^{s_1p}})}+ \Big(\frac{1}{R^{N}} \int_{-\frac{7}{8}R^{s_1p}}^{0} [u(\cdot,t)]_{W^{s_1,p}(B_R)}^p dt \Big)^\frac{1}{p}+1 \nonumber\\
		&\ +\sum_{(\el,s)}T^\el_{\infty,\el s}(u;x_0,R,-R^{s_1p},0)^\frac{1}{\el-1}+\|f\|^\frac{1}{p-1}_{L^\infty(Q_{R,R^{s_1p}})}\Big].
	\end{align}
\end{Theorem}
\begin{proof}
	Let $\sg\in [-R^{s_1p}(1-\mc K_R^{2-p}),0]$ and define
	\begin{align*}
		u_{R,\sg}(x,t):=\frac{1}{\mc K_R}u(Rx, \mc K_R^{2-p}R^{s_1p}t+\sg) \quad\mbox{for }(x,t)\in B_2\times(-2,0].
	\end{align*}
	By the scaling property of the operators involved in problem \eqref{probM}, it is easy to observe that $u_{R,\sg}$ solves
	\begin{align*}
		\pa_t u_{R,\sg}+(-\De)_p^{s_1} u_{R,\sg}+\la (-\De)_{q,\tl a}^{s_2}u_{R,\sg}=\tl f \quad\mbox{in }B_2\times(-2,0],
	\end{align*}
	where $\la=\mc K_R^{q-p}R^{s_1p-s_2q}$ with the re-scaled functions $\tl a(x,y)=a(Rx,Ry)$ and $\tl f(x,t)=\frac{R^{ps_1}}{\mc K_R^{p-1}}f\big(Rx,\frac{R^{s_1p}}{\mc K_R^{p-2}}t+\sg\big)$. Moreover, $u_{R,\sg}$ satisfies 
	\begin{align*}
		&\|u_{R,\sg}\|_{L^\infty(B_1\times[-1,0])}\leq 1, \ \ \sum_{(\el,s)}T^\el_{\infty,\el s}(u_{R,\sg};0,1,-1,0)^\frac{1}{\el-1}\leq 1 \quad\mbox{and}\nonumber\\
		&\int_{-\frac{7}{8}}^{0} [u_{R,\sg}]_{W^{s_1,p}(B_1)}^p dt\leq 1.
	\end{align*}
	We claim that 
	\begin{align}\label{eqi50}
		\sup_{t\in[-1/2,0]} [u_{R,\sg}(\cdot,t)]_{C^{0,\al}(B_{1/2})} \leq C(1+\la+\|\tl f\|_{L^\infty(B_1\times(-1,0])})^{i_\infty},
	\end{align}
	where $i_\infty\in \mb N$ is such that $i_\infty>\frac{3\al+(N-2)s_1}{s_1-\al}$ and $C$ is a constant independent of $R$ and $\sg$.
	Indeed, the proof follows by using Proposition \ref{propin1} and following the iterative scheme presented as in \cite[Theorem 4.2]{brascoP} together with the observation that the process runs for $i_\infty$ steps only. Scaling back to $u$, from \eqref{eqi50}, we get 
	\begin{align*}
		&\sup_{\sg-\frac{1}{2}\mc K_R^{2-p}R^{s_1p}\leq t\leq 0} [u(\cdot,t)]_{C^{0,\al}(B_{R/2})} \nonumber\\
		&\quad\leq \frac{C}{R^\al}(1+\la+\mc K_R^{1-p}R^{ps_1}\|f\|_{L^\infty(B_R\times(-R^{ps_1},0])})^{i_\infty}\mc K_R,
	\end{align*}
	where $C=C(N,s_1,p,q,\|a\|_\infty,\al)>0$ is a constant. Note that in the above expression $C$ does not depend on $\sg$ and noting $\sg\in [-R^{s_1p}(1-\mc K_R^{2-p}),0]$ with $0<R<1\leq \mc K_R$, we get 
	\begin{align*}
		\sup_{-\frac{R^{s_1p}}{2}\leq t\leq 0} [u(\cdot,t)]_{C^{0,\al}(B_{R/2})}  \leq \frac{C}{R^{\al}} \mc K_R^{(q-p)i_\infty+1}.
	\end{align*}
	This proves the theorem by recalling the definition of $\mc K_R$.
\end{proof}

\begin{Proposition}\label{propin2}
	Let the hypotheses of Proposition \ref{propin1} hold. Additionally suppose that 
	for some $\ka\geq 2$, $\nu<1$ and $0<h_0<\frac{1}{10}$, we have $(1+\nu\ka)/\ka<1$ and
	\begin{align*}
		\int_{T_0}^{T_1}\sup_{0<|h|<h_0} \bigg\| \frac{\de^2_h u(\cdot,t)}{|h|^{s_1}}\bigg\|^m_{L^m(B_{R+4h_0})} dt<\infty,
	\end{align*}
	for $4h_0<R\leq 1 - 5h_0$ and $-1<T_0<T_1\leq 0$. Then, for all $\mu\in (0,T_1-T_0)$, there holds: 
	\begin{align*}
		&\int_{T_0+\mu}^{T_1}\sup_{0<|h|<h_0} \bigg\| \frac{\de^2_h u(\cdot,t)}{|h|^\frac{1+s_1p+\nu\ka}{\ka-1+p}}\bigg\|^{\ka-1+p}_{L^{\ka-1+p}(B_{R-4h_0})} dt \nonumber\\
		&\quad+ \frac{1}{\ka+1} \sup_{0<|h|<h_0} \bigg\| \frac{\de_h u(\cdot,T_1)}{|h|^{\frac{1+\nu\ka}{\ka+1}}}\bigg\|^{\ka+1}_{L^{\ka+1}(B_{R-4h_0})} \\
		&\leq C(1+\la+\|f\|_{L^\infty(Q_1)}) \int_{T_0}^{T_1} \left( \sup_{0<|h|<h_0} \bigg\| \frac{\de^2_h u(\cdot,t)}{|h|^\frac{1+\nu\ka}{\ka}}\bigg\|^\ka_{L^\ka(B_{R+4h_0})} +1 \right)dt,
	\end{align*}
	where $C= C(N,h_0,p,q,\ka,s_1,\|a\|_\infty)>0$ (which depends inversely on $h_0$). 
\end{Proposition}
\begin{proof}
	For every $\e>0$, by Theorem \ref{thmins}, we have $u\in C^{0,s_1-\e}_{x,{\rm loc}}(B_2\times (-2,0])$ and
	\begin{align*}
		\sup_{t\in [T_1,T_1]} [u(\cdot,t)]_{C^{0,s_1-\e}(B_{R+h_0})} \leq C(N,h_0,p,s_1,q,s_2,\|a\|_\infty).
	\end{align*}
	Consequently, for $\e\in \Big(0,\min\{\frac{2(1-s_1)}{p-2},s_1 \}\Big)$, using the bounds on $\eta$ and $\na\eta$, we obtain
	\begin{align*}
		&\int_{B_R}a_\el\frac{|u(x)-u(y)|^{\el-2}|\eta^\frac{p}{2}(x)-\eta^\frac{p}{2}(y)|^2}{|x-y|^{N+s\el}}dy \\
		&\leq C\int_{B_R}|x-y|^{-N-s\el+2+(\el-2)(s_1-\e)}dy\leq C,
	\end{align*}
	where we have used $2+(\el-2)(s_1-\e)-s\el>0$. Therefore, with the notation of Proposition \ref{propin1}, we have 
	\begin{align*}
		|\mb I_{5,p}(t)+\mb I_{4,q}(t)|\leq C(1+\la)\int_{B_R}\frac{|\de_h u(x,t)|^{\ka+1}}{|h|^{1+\nu\ka}}dx\leq C\int_{B_R}\frac{|\de_h u(x,t)|^{\ka}}{|h|^{1+\nu\ka}}dx,
	\end{align*}
	where we used the local boundedness of $u$.  The other terms can be estimated similarly. Then, proceeding as in \cite[Proposition 5.1]{brascoP}, we can complete the proof of the proposition. 
\end{proof}
Using Proposition \ref{propin2} and following the iteration argument as in the proof of \cite[Theorem 5.2]{brascoP}, we can prove the following higher space regularity result.
\begin{Theorem}\label{thminwtl}
	Let the hypothesis of Theorem \ref{thmins} be true. 
	Then, for every $\al\in (0,\Theta)$, $u\in C^{0,\al}_{x, {\rm loc}}(\Om\times I)$, where $\Theta$ is given by \eqref{eqTheta}. More precisely, for every $\al\in (0,\Theta)$ and $R\in (0,1)$ satisfying $Q_{2R,2R^{s_1p}}\Subset\Om\times I$, there exists a constant $C=C(N,s_1,s_2,p,q,\al,\|a\|_\infty)>0$ such that 
	\begin{align}\label{eqinh}
		\sup_{t\in[-\frac{R^{s_1p}}{2},0]} [u(\cdot,t)]_{C^{0,\al}(B_\frac{R}{2})} \leq \frac{C}{R^{\al}} \mc K_R^{j_\infty}, 
	\end{align}
	for some $j_\infty\in\mb N$ depending only on $N,p,s_1$ and $\al$, where the constant $\mc K_R$ is given by \eqref{K R}.
\end{Theorem}

\subsection{Regularity in time}
To move forward, for $u\in L^1(Q_{r,\theta}(x_0,t_0))$, we set 
\begin{align*}
	(u)_{(x_0,t_0),r,\theta}=\Xint-_{Q_{r,\theta}(x_0,t_0)} u(x,t)dxdt,
\end{align*}
and when the center is clear, we denote it by $(u)_{r,\theta}$. 
\begin{Proposition}\label{proptime1}
	Let $u$ be a local weak solution to problem \eqref{probM} in $B_2\times (-2,0]$ with $f\in L^\infty(B_1\times (-1,0])$ such that 
	\begin{align}\label{eqt91}
		\|u\|_{L^\infty(\mb R^N\times[-1,0])}\leq 1 \quad\mbox{and } \sup_{t\in[-\frac{1}{2},0]}[u(\cdot,t)]_{C^{0,\al}(B_\frac{1}{2})}\leq K_\al, 
	\end{align}
 for $\al\in (s_1,\Theta)$,	where $\Theta$ is defined in \eqref{eqTheta}. Then, there exists a positive constant $C$ depending only on $N,s_1,p,q,K_\al,\al,\|a\|_\infty$ and $\|f\|_{L^\infty(B_1\times (-1,0])}$ such that 
	\begin{align*}
		|u(x,t)-u(x,\tau)|\leq C|t-\tau|^\ga, \quad\mbox{for all }x\in B_{1/4}, \mbox{ and } t,\tau\in (-1/4,0],
	\end{align*}
	where $\ga=\frac{\al}{ps_1-(p-2)\al}$. In particular, $u\in C^{0,\ga}_{t}(Q_{\frac{1}{4},\frac{1}{4}})$ for any $\ga<\Ga$, where $\Ga$ is given by \eqref{eqGamma}.
\end{Proposition}
\begin{proof}
	%We proceed similar to \cite[Proposition 6.2]{brascoP}. 
	For fixed $(x_0,t_0)\in Q_{1/4,1/4}$, we choose $r\in (0,1/8)$ and $\theta\in (0,1/8)$. Then, we observe that 
	\begin{align*}
		Q_{r,\theta}(x_0,t_0):=B_r(x_0)\times(t_0-\theta,t_0]\subset B_{3/8}\times(-1/2,0].
	\end{align*}
	Let $\psi\in C^\infty_{c}(B_{r/2}(x_0))$ be a non-negative cut-off function such that 
	\begin{align*}
		\psi\equiv \|\psi\|_{L^\infty(B_\frac{r}{2}(x_0))} \ \mbox{ in } B_\frac{r}{4}(x_0), \quad (\psi)_{r}=1 \quad\mbox{and } \|\na\psi\|_{L^\infty(B_\frac{r}{2}(x_0))}\leq \frac{C}{r},
	\end{align*}
	for some $C=C(N)>0$ (note that the fact $(\psi)_{r}=1$ yields $\|\psi\|_{L^\infty(B_\frac{r}{2}(x_0))} \leq 4^N$). Set 
	\begin{align*}
		(u\psi)_r(t):=\Xint-_{B_r(x_0)}u(x,t)\psi(x)dx.
	\end{align*}
	Then, on account of \cite[Equ. (6.3)]{brascoP}, we observe that 
	\begin{align*}
		\Xint-_{Q_{r,\theta}(x_0,t_0)} |u(x,t)-(u)_{r,\theta}|dxdt &\leq 2 \Xint-_{Q_{r,\theta}(x_0,t_0)} |u(x,t)-(u\psi)_{r}(t)|dxdt \nonumber\\
		&\quad+2\Xint-_{Q_{r,\theta}(x_0,t_0)} |(u\psi)_{r}(t)-(u\psi)_{r,\theta}|dxdt \nonumber\\
		&=:\mc O_1+\mc O_2.
	\end{align*}
	Using the Poincar\'e inequality of \cite[Lemma 6.1]{brascoP}, we obtain
	\begin{align*}
		\mc O_1&\leq 2\Bigg(\Xint-_{Q_{r,\theta}(x_0,t_0)} |u(x,t)-(u\psi)_{r}(t)|^pdxdt \Bigg)^{1/p} \\
		&\leq C\Bigg( \frac{r^{s_1p}}{|Q_{r,\theta}(x_0,t_0)|} \int_{t_0-\theta}^{t_0}\int_{B_r(x_0)}\int_{B_r(x_0)} \frac{|u(x,t)-u(y,t)|^p}{|x-y|^{N+ps_1}}dxdydt\Bigg)^{1/p},
	\end{align*}
	for some $C=C(N,p,s_1)>0$. Note that for $x\in B_r(x_0)$, we have $B_r(x_0)\subset B_{2r}(x)\subset B_{1/2}$. Thus, using \eqref{eqt91}, we get (note that $\al>s_1$)
	\begin{align*}
		&\int_{B_r(x_0)}\int_{B_r(x_0)} \frac{|u(x,t)-u(y,t)|^p}{|x-y|^{N+ps_1}}dxdy \\
		&\leq K_\al^p \int_{B_r(x_0)} \left(\int_{B_r(x_0)} |x-y|^{-N+p(\al-s_1)}dy\right)dx \\
		&\leq CK_\al^p \, r^{(\al-s_1)p},
	\end{align*}
	where $C=C(N,s_1,p)$. Consequently,
	\begin{align*}
		\mc O_1\leq CK_\al r^{\al}.
	\end{align*}
	To estimate $\mc O_2$, we observe that
	\begin{align*}
		\mc O_2\leq 2\sup_{T_0,T_1\in (t_0-\theta,t_0]} |(u\psi)_r(T_0)-(u\psi)_r(T_1)|.
	\end{align*} 
	Thus, for $T_0,T_1\in (t_0-\theta,t_0]$ with $T_0<T_1$, from the weak formulation of problem \eqref{probM}, we obtain
	\begin{align*}
		&|(u\psi)_r(T_0)-(u\psi)_r(T_1)|\\ &=\frac{1}{|B_r(x_0)|}\bigg|\int_{B_r(x_0)}u(x,T_0)\psi(x)dx-\int_{B_r(x_0)}u(x,T_1)\psi(x)dx \bigg| \\
		&=\frac{1}{|B_r(x_0)|}\bigg|\int_{T_0}^{T_1}\iint_{\mb R^{2N}} [u(x,t)-u(y,t)]^{p-1}(\psi(x)-\psi(y))d\mu_1dt \\
		&\ + \int_{T_0}^{T_1} \iint_{\mb R^{2N}} a(x,y) [u(x,t)-u(y,t)]^{q-1}(\psi(x)-\psi(y))d\mu_2dt \\
		& \ -\int_{T_0}^{T_1}\int_{B_r(x_0)}f(x,t)\psi(x)dxdt \bigg|\\
		&\leq \bigg|\sum_{(\el,s)}\int_{T_0}^{T_1}\int_{B_r(x_0)}\Xint-_{B_r(x_0)} a_\el[u(x,t)-u(y,t)]^{\el-1}(\psi(x)-\psi(y))d\mu dt \bigg|\\
		&\ + 2\bigg|\sum_{(\el,s)}\int_{T_0}^{T_1}\int_{\mb R^N\setminus B_r(x_0)}\Xint-_{B_{r}(x_0)}a_\el [u(x,t)-u(y,t)]^{\el-1}\psi(x)d\mu dt \bigg| \\
		&\  + \bigg|\int_{T_0}^{T_1}\Xint-_{B_r(x_0)}f(x,t)\psi(x)dxdt \bigg|\\
		&=:\mathfrak{J}_1+\mathfrak{J}_2+\mathfrak{J}_3.
	\end{align*}
	Using H\"older's inequality, \eqref{eqt91} and the bounds on $\psi$, we get (noting $\al>\max\{s_1, s_2\}$)
	\begin{align*}
		&\mathfrak{J}_1\\&\leq \frac{\|a_\el\|_{\infty}}{|B_r(x_0)|} [\psi]_{W^{s,\el}(B_r(x_0))} \int_{t_0-\theta}^{t_0} \Bigg(\int_{B_r(x_0)}\int_{B_r(x_0)} |u(x,t)-u(y,t)|^\el d\mu \Bigg)^\frac{\el-1}{\el}dt \\
		&\leq CK_\al^{\el-1} r^{-s} \int_{t_0-\theta}^{t_0} \Bigg(\int_{B_r(x_0)}\int_{B_r(x_0)} |x-y|^{\al\el-N-s\el}dxdy \Bigg)^\frac{\el-1}{\el}dt \\
		&\leq C \; K_\al^{\el-1} \; \theta \; r^{\al(\el-1)-\el s},
	\end{align*}
	where $C=C(N,p,s_1,\al,\|a\|_{\infty})>0$. To estimate $\mathfrak{J}_2$,  we observe that for $x\in B_r(x_0)$, $t\in [-1/2,0]$ and $y\in\mb R^N$, we have
	\begin{align*}
		|u(x,t)-u(y,t)|\leq C|x-y|^\al \quad\mbox{for some }C=C(K_\al,\al)>0.
	\end{align*} 
	Indeed, for $y\in B_{1/2}$ it follows from \eqref{eqt91} that if $y\in \mb R^N\setminus B_{1/2}$, then 
	\begin{align*}
		|u(x,t)-u(y,t)|\leq 2\|u\|_{L^\infty(\mb R^N\times[-1,0])}\leq 2\leq \frac{2}{8^\al}|x-y|^\al.
	\end{align*} 
	Therefore (noticing that ${\rm supp}(\psi)\subset B_{r/2}(x_0)$),
	\begin{align*}
		&\mathfrak{J}_2\\
		&\leq 2(T_1-T_0)\|a\|_{\infty}\|\psi\|_{L^\infty(B_\frac{r}{2}(x_0))}  \int_{B_r(x_0)^c} \Xint-_{B_\frac{r}{2}(x_0)} \frac{|u(x,t)-u(y,t)|^{\el-1}}{|x-y|^{N+s\el}}dxdy\\
		&\leq C\theta \int_{B_r(x_0)^c}\Xint-_{B_r(x_0)} |x-y|^{\al(\el-1)-N-s\el}dxdy\\
		&\leq C\; \theta \; r^{\al(\el-1)-\el s},
	\end{align*}
	where $C=C(N,p,s_1,\al,\|a\|_{\infty})>0$. Finally, noting the fact that $r\in (0,1/8)$ and $\al(p-1)-ps_1<0$, we have 
	\begin{align*}
		\mathfrak{J}_3 \leq C \|f\|_{L^\infty(Q_{r,\theta}(x_0,t_0))} \theta \leq C\|f\|_{L^\infty(Q_{r,\theta}(x_0,t_0))}\theta r^{\al(p-1)-ps_1}.
	\end{align*}
	Consequently,
	\begin{align*}
		\mc O_2 \leq  C\; \theta \; r^{\al(\el-1)-\el s}.
	\end{align*}
	Finally, on account of the relation $\al(p-1)-ps_1\leq \al(q-1)-qs_2$ and the fact $r\in (0,1/8)$, we have
	\begin{align*}
		\Xint-_{Q_{r,\theta}(x_0,t_0)} |u(x,t)-(u)_{r,\theta}|dxdt	\leq C r^\al+ C\; \theta \; r^{\al(p-1)-p s_1},
	\end{align*}
	where  $C=C(N,p,s_1,\al,K_\al,\|a\|_{\infty},\|f\|_{L^\infty(Q_{r,\theta}(x_0,t_0))})>0$. Now, the rest of the proof can be completed as in the proof of \cite[Proposition 6.2]{brascoP}. 
\end{proof}
\begin{Remark}\label{remwtglbd}
	It is clear from the proof that when $u\in L^\infty(\mb R^N\times[-1,0])$ is not assumed, then only the proof of $\mf J_2$ is changed. Indeed, it can be shown that 
	\begin{align*}
		\mf J_2\leq C\;\theta\;r^{-\el s}.
	\end{align*}
	Then the proof of the time regularity can be completed by proceeding similarly and choosing $\theta$ accordingly. 
\end{Remark}
\textbf{Proof of Theorem \ref{thminh}:} The proof of the theorem follows from  Theorem \ref{thminwtl} and  Proposition \ref{proptime1} after rescaling (see \cite[Section 7]{brascoP}). 
Additionally,  by noting $u\in L^\infty(\mb R^N\times[\hat t_0-R^{s_1p},\hat t_0])$ and using Corollary \ref{corcaccpp}, we see that 
\begin{align*}
	\mc K_R\leq C  \Big(\|u\|_{L^\infty(\mb R^N\times[\hat t_0-R^{s_1p},\hat t_0])}+1+\|f\|^\frac{1}{p-1}_{L^\infty(Q_{R,R^{s_1p}})}\Big)^{q/p}.
\end{align*} 
This completes the proof of the theorem.

\section{A further investigation of the stationary case}
In this section, we obtain higher H\"older continuity results for weak solutions to problem \eqref{probSt} with more general modulating coefficients and non-homogeneous terms.
\subsection{Interior regularity}
We first start with the following stability result. 
\begin{Lemma}\label{lemstab}
	Suppose that $p\geq 2$, $q< p^*_{s_1}$ and $qs_2\leq ps_1$. Let $u$ be a weak solution to problem \eqref{probSt} in $B_2\equiv B_{2}(x_0)\Subset\Om$,   such that   
	\begin{align*}
		\|u\|_{L^\infty(B_{{3/2}})} +\sum_{(\el,s)}
		T^\el_{\el s}(u;x_0,{3/2})\leq M,
	\end{align*}
	for some $M\geq 1$. Let $\de\in (0,1)$ be such that 
	\begin{align*}
		\| a - \tl a\|_{L^\infty(\mb R^N\times\mb R^N)} \leq \de\quad\mbox{and}\quad \|f\|_{L^\Bbbk(B_{3/2})}\leq\de,
	\end{align*}
	where $\tl a(\cdot,\cdot)\in L^\infty(\mb R^N\times\mb R^N)$ is a non-negative function satisfying the assumption {\rm\textbf{(A.1)}} in $B_1\times B_1$ with
	\begin{equation*}%\label{eqi aux V}
		\left\{
		\begin{array}{rllll}
			(-\De)_p^{s_1}v+(-\De)_{q,\tl a}^{s_2}v&=0 \quad\mbox{in }B_1,\\
			v&=u \quad\mbox{in }\mb R^N\setminus B_1.
		\end{array}
		\right.
	\end{equation*}
	Then, there exists $\vep(\de)=\vep(\de,N,p,q,s_1,\Bbbk,M,\|a\|_{\infty})>0$ such that
	\begin{align*}
		\| u - v\|_{L^\infty(B_{3/8})}<\vep(\de)
	\end{align*}
	and $\vep(\de)$ converges to $0$ as $\de\to 0$.
\end{Lemma}
\begin{proof}
	From  Lemma \ref{localbdd}, we see that  $u$ and $v$ are locally bounded. Then, the triangle inequality and the relation $\|a\|_\infty-\de\leq \|\tl a\|_\infty\leq \|a\|_\infty+\de$ prove the existence of  $\vep(\de)$. 
	To prove the convergence of $\vep(\de)$ with respect to $\de$, we proceed by the method of contradiction.  
	Suppose that there exist sequences $\{a_m\}$, $\{\tl a_m\}$,  $\{u_m\}$, $\{f_m\}$ and $\{v_m\}$ with the following properties:
	\begin{align}\label{eqi56}
		&(-\De)_p^{s_1}u_m+(-\De)_{q,a_m}^{s_2}u_m=f_m\quad\mbox{in }B_{2}, \\
		&\|u_m\|_{L^\infty(B_{{3/2}})} + \sum_{(\el,s)}
		T^\el_{\el s}(u_m;x_0,{3/2})\leq M, \label{eqi52} \\
		&\| a_m -\tl a_m \|_{L^\infty(\mb R^N\times\mb R^N)}\leq 1/m,\quad \|f_m\|_{L^\Bbbk(B_{3/2})}\leq 1/m \label{eqi53}
	\end{align}
	and 
	\begin{equation}\label{eqi59}
		\begin{array}{rllll}
			(-\De)_p^{s_1}v_m+(-\De)_{q,\tl a_m}^{s_2}v_m=0 \quad\mbox{in }B_1;\quad
			v_m=u_m \quad\mbox{ in }\mb R^N\setminus B_1,	
		\end{array}
	\end{equation}
	but 
	\begin{align}\label{eqi58}
		\liminf_{m\to\infty}\| u_m - v_m\|_{L^\infty(B_{3/8})}> 0.
	\end{align}
	%Taking $w_m=u_m-v_m$ as a test function, we have 
	% \begin{align*}
		% 	& \int_{\mb R^{N}} \int_{\mb R^N} ([u_m(x)-u_m(y)]^{p-1}-[v_m(x)-v_m(y)]^{p-1})(w_m(x)-w_m(y))d\mu_1 \nonumber\\ 
		% 	&+ \int_{\mb R^{N}} \int_{\mb R^N} \Big(a_m(x,y)[u_m(x)-u_m(y)]^{q-1}-\tl a_m(x,y) [v_m(x)-v_m(y)]^{q-1}\Big) (w_m(x)-w_m(y))d\mu_2 \nonumber\\
		% 	&=\int_{B_{\frac{3}{2}R}}f_m(x)w_m(x)dx.
		% \end{align*}
	% Consequently,
	Note that on account of \eqref{eqi52} and $qs_2\leq ps_1$, we have
	\begin{align*}
		&\int_{B_\frac{3}{2}}\int_{B_\frac{3}{2}} \tl a_m \frac{|u_m(x)-u_m(y)|^q}{|x-y|^{N+qs_2}}dxdy \\
		&\leq c\|u_m\|_{L^\infty(B_{\frac{3}{2}})}^{q-p} \int_{B_\frac{3}{2}}\int_{B_\frac{3}{2}} \frac{|u_m(x)-u_m(y)|^p}{|x-y|^{N+ps_1}}dxdy<\infty.
	\end{align*} 
	Consequently $u_m\in X_{u_m,\tl a_m}(B_1, B_{3/2})$ and $w_m:=u_m-v_m\in X_{0,\tl a_m}(B_1, B_{3/2})$. 
	Additionally,  since $v_m\in \mc W_{\tl a_m}(B_{3/2})$ and $u_m\in L^\infty(B_{3/2})$, proceeding similar to the proof of Lemma \ref{localbdd}, we can show that $v_m\in L^\infty(B_1)$ (however the bound may depend on $m$). As before, we see that $w_m\in \mc W_{a_m}(B_{3/2})$ with support contained in $B_1$.
	Therefore taking $w_m=u_m-v_m$ as a test function in the weak formulation of \eqref{eqi56} and \eqref{eqi59}, and using the monotonocity of the fractional operators (see \eqref{eqmon}), we deduce that
	\begin{align}\label{eqi51}
		&c_p\iint_{\mb R^{2N}}  |w_m(x)-w_m(y)|^{p}d\mu_1+c_q  \iint_{\mb R^{2N}}  \tl a_m(x,y)|w_m(x)-w_m(y)|^{q}d\mu_2 \nonumber\\
		&\leq  \iint_{\mb R^{2N}}  ([u_m(x)-u_m(y)]^{p-1}-[v_m(x)-v_m(y)]^{p-1})(w_m(x)-w_m(y))d\mu_1  \nonumber\\ 
		&\ + \iint_{\mb R^{2N}} \tl a_m(x,y) \Big([u_m(x)-u_m(y)]^{q-1}- [v_m(x)-v_m(y)]^{q-1}\Big) \nonumber\\
		&\qquad\qquad\quad\times(w_m(x)-w_m(y))d\mu_2\nonumber\\
		&=\iint_{\mb R^{2N}} (\tl a_m(x,y)-a_m(x,y)) [u_m(x)-u_m(y)]^{q-1}(w_m(x)-w_m(y))d\mu_2 \nonumber\\
		&\quad+\int_{B_2}f_mw_mdx  \nonumber\\
		&\leq \|a_m-\tl a_m\|_{\infty}  \iint_{\mb R^{2N}} |u_m(x)-u_m(y)|^{q-1}|w_m(x)-w_m(y)|d\mu_2 \nonumber\\
		&\quad+\int_{B_2}|f_m(x)w_m(x)|dx.
	\end{align} 
	To estimate the above integrals, we set
	\begin{align*}
		&\mb J_1:=\int_{B_{\frac{6}{5}}}\int_{B_{\frac{6}{5}}} |u_m(x)-u_m(y)|^{q-1}|w_m(x)-w_m(y)|d\mu_2,  \\
		&\mb J_2:=\int_{B_{\frac{6}{5}}}\int_{\mb R^N\setminus B_{\frac{6}{5}}} |u_m(x)-u_m(y)|^{q-1}|w_m(x)|d\mu_2 \quad\mbox{and} \\
		&\mb J_3:=\int_{B_2}f_m(x)w_m(x)dx.
	\end{align*} 
	Using \eqref{eqi52}, the fact $qs_2\leq ps_1$ and H\"older's inequality, we get 
	\begin{align*}
		\mb J_1&\leq C_q\|u_m\|_{L^\infty(B_{\frac{6}{5}})}^{q-p}   \int_{B_{\frac{6}{5}}}\int_{B_{\frac{6}{5}}} \frac{|u_m(x)-u_m(y)|^{p-1}|w_m(x)-w_m(y)|}{|x-y|^{N+s_1p}}dxdy \\
		&\leq C M^{q-p}[u_m]^{p-1}_{W^{s_1,p}(B_{\frac{6}{5}})}   \Bigg(  \int_{B_{\frac{6}{5}}}\int_{B_{\frac{6}{5}}}\frac{|w_m(x)-w_m(y)|^p}{|x-y|^{N+s_1p}}dxdy \Bigg)^\frac{1}{p}.
	\end{align*}
	%where we have used the fact that $qs_2\leq ps_1$ and the fractional Sobolev embedding result (see \cite{nezzaH}). 
	Moreover, using Caccioppoli type relation of Lemma \ref{caccpstn}, for the choice  $\psi\in C^\infty_c(B_{7/5})$ satisfying $0\leq\psi\leq 1$ and $\psi\equiv 1$ in $B_{6/5}$, we see that 
	\begin{align*}
		&[u_m\psi^\frac{q}{p}]^p_{W^{s_1,p}(B_{\frac{7}{5}})}\nonumber \\
		&\leq  C  \int_{B_{\frac{3}{2}}}\int_{B_{\frac{3}{2}}}  \frac{|\psi(x)-\psi(y)|^p}{|x-y|^{N+ps_1}} dxdy+C  \int_{\mb R^N\setminus B_{\frac{3}{2}}} \frac{|u_m(y)|^{p-1}+1}{|x_0-y|^{N+ps_1}}dy\\
		& \ +C\Big(\int_{B_\frac{7}{5}}|f_m|^\Bbbk\Big)^\frac{p}{(p-1)\Bbbk}+ C  \int_{B_{\frac{3}{2}}}\int_{B_{\frac{3}{2}}} a_m(x,y) \frac{|\psi(x)-\psi(y)|^q}{|x-y|^{N+qs_2}} dxdy \nonumber\\
		& \ + C    \int_{\mb R^N\setminus B_{\frac{3}{2}}}a_{m,2}(y) \frac{|u_m(y)|^{q-1}+1}{|x_0-y|^{N+qs_2}}dy \nonumber\\
		&\leq CM(1+\|a\|_\infty),
	\end{align*}
	where we have set $a_{m,2}(y):=\sup_{x\in B_{2}}a_m(x,y)$ and used \eqref{eqi52} along with H\"older's and Young's inequalities for the term involving $f_m$ (see the estimate of $\mb J_3$ for details). This yields that $[u_m]_{W^{s_1,p}(B_{6/5})}$ is bounded independent of $m$.
	Consequently,  
	\begin{align}\label{eqi54}
		\mb J_1\leq C \Bigg(  \iint_{\mb R^{2N}} \frac{|w_m(x)-w_m(y)|^p}{|x-y|^{N+s_1p}}dxdy \Bigg)^\frac{1}{p}.
	\end{align}
	Furthermore, since $w_m=0$ in $\mb R^N\setminus B_1$, \eqref{eqi52} and the fractional Poincar\'e inequality imply
	\begin{align}\label{eqi55}
		\mb J_2 &\leq C \Bigg[\int_{\mb R^N\setminus B_{\frac{6}{5}}} \frac{dy}{|x-y|^{N+qs_2}}  + \int_{\mb R^N\setminus B_{\frac{6}{5}}} \frac{|u_m(y)|^{q-1}}{|x-y|^{N+qs_2}}dy \Bigg]\int_{B_{1}} |w_m(x)|dx \nonumber\\
		&\leq CM  \Bigg( \int_{B_{1}}|w_m(x)|^pdx \Bigg)^\frac{1}{p} \leq C  \Bigg(\int_{\mb R^N}\int_{\mb R^N} \frac{|w_m(x)-w_m(y)|^p}{|x-y|^{N+s_1p}}dxdy \Bigg)^\frac{1}{p}.
	\end{align}
	Similarly, noting $(p^*_{s_1})'< \Bbbk$ and using \eqref{eqi53} together with the Sobolev embedding, we have
	\begin{align}\label{eqi68}
		\mb J_3\leq \int_{B_{\frac{3}{2}R}}|f_mw_m|dx &\leq C \Big(\int_{B_R}|f_m(x)|^{(p^*_{s_1})'}dx\Big)^\frac{1}{(p^*_{s_1})'} \Big(\int_{B_R}|w_m(x)|^{p^*_{s_1}}dx\Big)^\frac{1}{p^*_{s_1}} \nonumber\\
		&\leq \frac{C}{m} \Bigg(  \iint_{\mb R^{2N}} \frac{|w_m(x)-w_m(y)|^p}{|x-y|^{N+s_1p}}dxdy \Bigg)^\frac{1}{p}.
	\end{align}
	Now combining \eqref{eqi54}, \eqref{eqi55}, \eqref{eqi68} 
	and \eqref{eqi53} with \eqref{eqi51}, we obtain
	\begin{align*}
		\Bigg( \iint_{\mb R^{2N}} \frac{|w_m(x)-w_m(y)|^p}{|x-y|^{N+s_1p}}dxdy \Bigg)^\frac{p-1}{p} \leq C \frac{1}{m}.
	\end{align*}
	Then, by the fractional Sobolev inequality, we have
	\begin{align}\label{eqi57}
		\int_{B_1} |w_m(x)|^{p^*_{s_1}}dx  \to 0 \quad\mbox{as }m\to\infty.
	\end{align}
	Invoking Proposition \ref{holdunspstn} and \eqref{eqi52}, we get that the sequence $\{u_m\}$ is uniformly bounded in $C^{0,\al_1}$ in $B_{1/2}$, for some $\al_1\in (0,s_1)$.  
	We next note that on account of the local boundedness result of Lemma \ref{localbdd} and Caccioppoli type result of Lemma \ref{caccpstn} for $v_m$ and estimating the tail term as in the proof of Proposition \ref{holdunspstn} with the help of \eqref{eqi57} and \eqref{eqi52}, there holds 
	\begin{align*}%\label{eqia66}
		&\|v_m\|_{L^\infty(B_{3/4})}+[v_m]_{W^{s_1,p}(B_{3/4})}+\sum_{(\el,s)}T^\el_{\el s}(v_m;0,\frac{3}{4})^\frac{1}{\el-1} \nonumber\\
		&\leq C \Big( \|v_m\|^\frac{q}{p}_{L^\vartheta(B_{\frac{3}{4}})}+  \|w_m\|^\frac{q}{p}_{L^\vartheta(B_{\frac{3}{4}})} +\sum_{(\el,s)}T^\el_{\el s}(v_m;0,\frac{3}{4})^\frac{1}{\el-1} +1+\|f\|_{L^\Bbbk(B_1)}^\frac{1}{p-1}\Big)\\
		&\leq M_1,
	\end{align*}
	where $\vartheta<p^*_{s_1}$ is given by Lemma \ref{localbdd} and $C,M_1>0$ are independent of $m$ and are non-decreasing functions of $\|a\|_\infty$.
	Consequently, from the stationary counterpart of Theorem \ref{thminwtl} and \eqref{eqinh}, we have 
	\begin{align*}%\label{eqi cont v_m}
		\|v_m\|_{C^{0,\al_1}(B_{3/8})} 
		\leq C M_1, 
	\end{align*}
	where $CM_1$ is independent of $m$ and is a non-decreasing function of $\|a\|_\infty$. Thus,
	we get that the sequence $\{v_m\}$ is also uniformly bounded in $C^{0,\al_1}$ norm in $B_{3/8}$. Therefore, by the Arzela-Ascoli theorem, up to a subsequence, $\{u_m-v_m\}$ converges uniformly to some function $w$ in $B_{3/8}$. This together with \eqref{eqi57} implies that $w\equiv 0$, that is, (up to a subsequence),
	\begin{align*}
		\lim_{m\to\infty} \|u_m-v_m\|_{L^\infty(B_{3/8})}=0,	
	\end{align*}
	which contradicts \eqref{eqi58}. This completes the proof of the lemma.
\end{proof}

\begin{Proposition}\label{propap2}
	Suppose that $p\geq 2$,  $q<p^*_{s_1}$ and $qs_2\leq ps_1$.  Let $u$ be a local weak solution to problem \eqref{probSt}  in $B_2\Subset\Om$, such that
	\begin{align*}
		\|u\|_{L^\infty(B_1)}+\sum_{(\el,s)}\int_{\mb R^N\setminus B_1} \frac{|u(x)|^{\el-1}}{|x|^{N+s\el}}dx\leq M,
	\end{align*} 
	with $M\geq 1$. For any $\e>0$, there exists $\de_1=\de_1(\e,N,p,s_1,q,\Bbbk,M,\|a\|_\infty)>0$ such that if 
	\begin{align}\label{eqi60}
		\| a - \tl a\|_{L^\infty(\mb R^N\times\mb R^N)} \leq \de_1\quad\mbox{and}\quad \|f\|_{L^\Bbbk(B_1)}\leq\de_1,
	\end{align}
	for some $0\leq\tl a(\cdot,\cdot)\in L^\infty(\mb R^N\times\mb R^N)$ satisfying the assumption {\rm\textbf{(A.1)}}, then $u\in C^{0,{\Theta_1}-\e}_{\rm loc}(\Om)$, where ${\Theta_1}$ is given by \eqref{eqThetaiot}.  
\end{Proposition}
\begin{proof} 
	We provide the proof in two steps.\\
	\textbf{Step 1}: \textit{Regularity at the origin}.
	In this step we will prove that for every $\e>0$ and $r\in (0,1/2)$, there exists $\de>0$ such that for the data of the proposition, 	
	\begin{align*}
		\sup_{x\in B_r} |u(x)-u(0)|\leq C_1 r^{{\Theta_1}-\e},
	\end{align*}
	where $C_1=C_1(N,p,s_1,s_2,q,\|a\|_\infty,M,\e)>0$ is a constant. Without loss of generality, we may assume that $u(0)=0$ (otherwise, we consider $\bar u=u-u(0)$ as the weak solution).
	For fixed $\e\in (0,{\Theta_1})$, it suffices to show that there exist $\rho<1/8$ and $\de>0$ such that if $f$, $a$, $\tl a$ and $u$ are as in the proposition, then for any $n\in\mb N\cup\{0\}$, we have 
	\begin{align}\label{eqi61}
		\sup_{x\in B_{\rho^n}} |u(x)|\leq M \rho^{n({\Theta_1}-\e)} \quad\mbox{and}\quad\sum_{(\el,s)}\int_{\mb R^N\setminus B_1} \bigg|\frac{u(\rho^nx)}{\rho^{n({\Theta_1}-\e)}}\bigg|^{\el-1}\frac{dx}{|x|^{N+s\el}}\leq M.
	\end{align}
	The claim is true for $n=0$ by the assumption. Let us assume that \eqref{eqi61} holds for some $n\in\mb N$, then we need to prove it for $n+1$ provided \eqref{eqi60} holds for $\de>0$ small. For $\rho\in(0,1)$, we define 
	\begin{align*}
		w_n(x):=\frac{u(\rho^nx)}{\rho^{n({\Theta_1}-\e)}}.
	\end{align*}
	By the scaling property of the fractional operators, it is easy to verify that $w_n$ satisfies the following equation:
	\begin{align*}
		(-\De)_p^{s_1}w_n + (-\De)_{q,a_n}^{s_2}w_n =f_n,
	\end{align*}
	where
	\begin{align*}
		&a_n(x,y)=\rho^{n(ps_1-qs_2+(q-p)({\Theta_1}-\e))} a(\rho^nx,\rho^ny) \quad \mbox{and}\\
		&f_n(x)= \rho^{n(ps_1-({\Theta_1}-\e)(p-1))} f(\rho^nx).
	\end{align*}
	Moreover, by the induction hypothesis and the fact $\rho<1$, we have
	\begin{align}\label{eqi66}
		\|w_n\|_{L^\infty(B_1)} +	\sum_{(\el,s)}\int_{\mb R^N\setminus B_1} \frac{|w_n(x)|^{\el-1}}{|x|^{N+s\el}}dx\leq M
	\end{align} 
	and since $\Bbbk>N/(ps_1)$, using  \eqref{eqThetaiot}, we also have 
	\begin{align}\label{eqif_n}
		\|f_n\|_{L^\Bbbk(B_1)}\leq \rho^{n(ps_1-({\Theta_1}-\e)(p-1)-N/\Bbbk)} \|f\|_{L^\Bbbk(B_{\rho^n})}\leq \tl\de.
	\end{align}
	Let $v_n$ be the solution to the following problem:
	\begin{equation*}
		(-\De)_p^{s_1}v_n+(-\De)_{q,\tl a_n}^{s_2}v_n=0 \quad\mbox{in }B_{2/3};\quad
		v_n=w_n \quad\mbox{in }\mb R^N\setminus B_{2/3},
	\end{equation*}
	where we have set 
	\begin{align*}
		\tl a_n(x,y)=\rho^{n(ps_1-qs_2+(q-p)({\Theta_1}-\e))} \tl a(\rho^nx,\rho^ny).
	\end{align*}
	On account of \eqref{eqi60} and the assumption $qs_2\leq ps_1$, we have 
	\begin{align*}%\label{eqia_n}
		\| a_n -\tl a_n \|_{L^\infty(\mb R^N\times\mb R^N)}\leq \de. %\quad\mbox{and}\quad \|f_n\|_{L^\Bbbk(B_1)}\leq \|f\|_{L^\Bbbk(B_1)}.
	\end{align*}
	Therefore, Lemma \ref{lemstab} yields
	\begin{align}\label{eqi67}
		\| v_n-w_n\|_{L^\infty(B_{1/4})}\leq \vep(\de),
	\end{align}
	where $\vep(\de)$ is independent of $n$, provided \eqref{eqi60} holds for small $\de>0$.
	%Then, using Lemma \ref{lem glbl bdd stbl} (thanks to \eqref{eqi66}), for $\vep>0$, we obtain  
	%\begin{align}\label{eqi67}
	%	\| v_n-w_n\|_{L^\infty(B_{1/4})}\leq \vep,
	%\end{align}
	%provided \eqref{eqif_n} holds for small $\tl\de>0$ given by Lemma \ref{lem glbl bdd stbl} and is independent of $n$. 
	This together with the fact $w_n(0)=0$,  implies that
	\begin{align}\label{eqiwkbd}
		|w_n(x)|&\leq |w_n(x)-v_n(x)|+|v_n(0)-w_n(0)|+|v_n(x)-v_n(0)| \nonumber\\
		&\leq 2\vep(\de)+[v_n]_{C^{0,{\Theta_1}-\e/2}(B_{1/8})}|x|^{{\Theta_1}-\e/2}, \quad\mbox{for all }x\in B_{1/8},
	\end{align}
	where we have used the $C^{0,{\Theta_1}-\e/2}$ regularity of $v_n$ from Theorem \ref{thminwtl}. 
	Next, from the stationary counterpart of \eqref{eqinh}, we have
	\begin{align*}
		&[v_n]_{C^{0,\Theta_1-\frac{\e}{2}}(B_{\frac{1}{8}})} \\
		&\leq C_4 \Big(\|v_n\|_{L^\infty(B_{\frac{1}{4}})}+   [v_n]_{W^{s_1,p}(B_{\frac{1}{4}})}  +1 +\sum_{(\el,s)}T^\el_{\el s}(v_n;0,\frac{1}{4})^\frac{1}{\el-1}\Big)^{j_\infty},
	\end{align*}
	where $C_4=C_4(N,p,s_1,q,\e,M,\|a\|_\infty)>0$ is a constant.   Using \eqref{eqi66} and \eqref{eqi67}, we obtain 
	\begin{align*}
		\|v_n\|_{L^\infty(B_{1/4})} \leq \|v_n-w_n\|_{L^\infty(B_{1/4})} + \|w_n\|_{L^\infty(B_{1/4})}\leq \vep+2M\leq 3M,
	\end{align*}
	provided \eqref{eqif_n} holds.
	For the tail terms, noting $q\leq p^*$, similar to \eqref{eqi57}, we have
	\begin{align*}
		\|v_n-w_n\|_{L^\el(B_{1/2})}\leq C_5,
	\end{align*}
	where $C_5>0$ is independent of $n$. Therefore, on account of \eqref{eqi66} and H\"older's inequality, we deduce that
	\begin{align*}
		&\int_{B_{1/4}^c}\frac{|v_n(y)|^{\el-1}}{|y|^{N+s\el}}dy\\
		&\leq C_6\Big(\int_{B_{1/2}^c}\frac{|w_n(y)|^{\el-1}}{|y|^{N+s\el}}dy+\|w_n\|^{\el-1}_{L^\infty(B_{1/2})}+\int_{B_{1/2}}|v_n(x)-w_n(x)|^{\el-1}dx \Big)\\
		&\leq C_7.
	\end{align*}
	Similar to Lemma \ref{lemstab}, using the Caccioppoli inequality of Lemma \ref{caccpstn}, we can show that the Sobolev seminorm of $v_n$ is bounded. 
	Consequently, from \eqref{eqiwkbd}, we get
	\begin{align}\label{eqi62}
		|w_n(x)|\leq 2\vep(\de)+C_9|x|^{{\Theta_1}-\e/2} \quad\mbox{in } B_{1/9},
	\end{align}
	where $C_9$ is a constant independent of $n$. Next, set
	\begin{align*}
		w_{n+1}(x):=\frac{w_n(\rho x)}{\rho^{{\Theta_1}-\e}} 
		=\frac{u(\rho^{n+1}x)}{\rho^{(n+1)({\Theta_1}-\e)}}.
	\end{align*}
	By choosing $\de>0$ small enough such that $2\vep(\de)<\rho^{\Theta_1}$, from \eqref{eqi62}, we obtain 
	\begin{align*}
		|w_{n+1}(x)|\leq 2\vep\rho^{\e-{\Theta_1}}+C_9\rho^{\e-{\Theta_1}}|\rho x|^{{\Theta_1}-\e/2} \leq (1+C_9|x|^{{\Theta_1}-\e/2})\rho^{\e/2} \quad\mbox{in }B_\frac{1}{8\rho}.
	\end{align*}
	In particular, choosing $\rho>0$ small enough, we get $\|w_{n+1}\|_{L^\infty(B_{1/(8\rho)})}\leq 1$. This on using the definition of $w_{n+1}$ and the fact that $8\rho<1$ proves the first part of  \eqref{eqi61}. For the second part, by noting the bound on ${\Theta_1}$ and proceeding similar to \cite[Proposition 6.2, pp. 838-839]{brascoH}, we can show that
	\begin{align*}
		\int_{\mb R^N\setminus B_1}\frac{|w_{n+1}(x)|^{\el-1}}{|x|^{N+s\el}}dx\leq \Big(\frac{C_2}{\e(\el-1)}+C_3+1\Big)M\rho^{\vep(\el-1)/2}\leq M,
	\end{align*}
	for appropriate choice of $\rho>0$. This completes the iteration process. 
	Here, we stress that the constant $\de>0$, small enough, is chosen such that $2\vep<\rho^{\Theta_1}$ and $\rho$ is also chosen small enough depending only on $N,p,s_1,s_2,q,\Bbbk,M,\e$ and the constant $S$ satisfying $\|a\|_\infty\leq S$.
	\\
	\textbf{Step 2}: {\it Regularity in the ball}.
	We now show the H\"older continuity of $u$ in the whole ball $B_{1/8}$.  Let $z\in B_{1/2}$ and define
	\begin{align*}
		v(x):=L^{-\frac{1}{p-1}}u({x/2}+z), \quad x\in\mb R^N \quad\mbox{with } L=2^{N+2}(1+|B_1|).
	\end{align*}
	Observe that $v$ is a weak solution  of 
	\begin{align*}
		(-\De)_p^{s_1}v +(-\De)_{q,\hat a}^{s_2}v =\widehat f \quad\mbox{in }B_{1},
	\end{align*}
	where 
	\begin{align*}
		\widehat a(x,y)= L^{\frac{q-p}{p-1}}2^{qs_2-ps_1}a\big(\frac{x}{2}+z,\frac{y}{2}+z\big) \quad\mbox{and }
		\widehat f(x)= 2^{-ps_1}L^{-1}f\big(\frac{x}{2}+z\big).
	\end{align*}
	For $0<\e<{\Theta_1}$, let $\de$ be as obtained in Step 1 for the choice $\widehat a$. 
	Assume further that 
	\begin{align*}
		\|a-\tl a\|_{L^\infty(\mb R^N\times\mb R^N)}\leq {2^{ps_1-qs_2}}{L^\frac{p-q}{p-1}}\de.
	\end{align*}
	Setting $\widehat{\tl a}(x,y)=L^{\frac{q-p}{p-1}}2^{qs_2-ps_1}\tl a({x/2}+z,{y/2}+z)$, we have
	\begin{align*}
		&\|\widehat a -\widehat {\tl a}\|_{L^\infty(\mb R^N\times\mb R^N)} \leq L^\frac{q-p}{p-1}2^{qs_2-ps_1} \|a-\tl a\|_{L^\infty(\mb R^N\times\mb R^N)}\leq \de \quad\mbox{and}\\
		&\|\widehat f\|_{L^\Bbbk(B_1)} \leq  \frac{2^{N/\Bbbk-ps_1}}{L} \|f\|_{L^\Bbbk(B_1)}\leq \de.   
	\end{align*}
	%Then,
	%\begin{align*}
	%	\|\widehat f\|_{L^\Bbbk(B_1)} \leq \frac{2^{N/\Bbbk-ps_1}}{L} \|f\|_{L^\Bbbk(B_1)}\leq \tl\de. 
	%\end{align*}
	By construction, we have
	\begin{align*}
		&\|v\|_{L^\infty(B_1)}\leq M \quad\mbox{and}\\
		&\int_{\mb R^N\setminus B_1} \frac{|v(x)|^{\el-1}}{|x|^{N+s\el}}dx\leq \frac{1}{L^\frac{\el-1}{p-1}} 	\int_{\mb R^N\setminus B_1} \frac{|u(x)|^{\el-1}}{|x|^{N+s\el}}dx+\frac{2^N|B_1|}{L^\frac{\el-1}{p-1}} \|u\|_{L^\infty(B_1)}^{\el-1}\leq M.
	\end{align*}
	Applying Step 1 to $v$, we get 
	\begin{align*}
		\sup_{x\in B_r}|v(x)-v(0)|\leq C_1 r^{{\Theta_1}-\e}, \quad r\in (0,1/2).
	\end{align*}
	This implies that, for any $z\in B_{1/2}$, 
	\begin{align}\label{eqi63}
		\sup_{x\in B_r}|u(x)-u(z)|\leq C_1 L^{\frac{1}{p-1}}r^{{\Theta_1}-\e}, \quad r\in (0,1/4).	
	\end{align}
	Next, we fix $x,y\in B_{1/8}$ and set $|x-y|=r$. We observe that $r<1/4$, then setting $z=(x+y)/2$, we get from \eqref{eqi63} that
	\begin{align*}
		|u(x)-u(y)|&\leq |u(x)-u(z)|+|u(y)-u(z)|\\
		&\leq 2\sup_{\xi\in B_r(z)}|u(\xi)-u(z)|\leq 2CL^\frac{1}{p-1}|x-y|^{{\Theta_1}-\e}.
	\end{align*}
	This proves the proposition for the choice $\de_1=\min\Big\{\de,2^{ps_1-qs_2}(2^{N+2}(1+|B_1|))^\frac{p-q}{p-1}\de\Big\}$.  
\end{proof}

\textbf{Proof of Theorem \ref{thmingenk}}:
Let $R_0\in (0,1)$ be fixed such that $B_{R_0}\equiv B_{R_0}(x_0)\Subset\Om$. For  $\al\in (0,{\Theta_1})$ and  $\e=\Theta_1-\al$, let the constant $\de_1=\de_1(\e,N,p,s_1,s_2,q,\|a\|_\infty)>0$ be given by Proposition \ref{propap2}. We set 
\begin{align*}%\label{mcmbd}
	\mc M=\mc M(R_0):=& \mc C\|u\|_{L^\infty(B_{R_0})} +\Big[\frac{R_0^{s_1p-\frac{N}{\Bbbk}}\|f\|_{L^\Bbbk(B_{R_0})}}{\de_1}\Big]^\frac{1}{p-1} \\
	&+\sum_{(\el,s)}2^\frac{N+s\el}{\el-1} T^{\el}_{\el s}(u;x_0,R_0)^\frac{1}{\el-1}+1,
\end{align*}
where $\mc C:=\max\Big\{1,\Big(\frac{N|B_1|}{ps_1}\Big)^{1/(p-1)},\Big(\frac{N|B_1|}{qs_2}\Big)^{1/(q-1)} \Big\}$.
We now choose $R\in (0,R_0/2)$ such that 
\begin{align}\label{eqia65}
	\mc M^{q-p}R^{ps_1-qs_2}\leq 1,
\end{align}
(thanks to the condition $ps_1>qs_2$).
%Let $\al\in (0,{\Theta_1})$ be fixed and for $\e=\Theta_1-\al$, let $\de_1=\de_1(\e,N,p,s_1,s_2,q,\|a\|_\infty)>0$ be given by Proposition \ref{propap1} for the choice $a(x,y)=\mc M^{q-p}R^{ps_1-qs_2}a(Rx,Ry)$, there. 
%Set $\de_2:=\frac{\de_1}{\mc M^{q-p}}$.
By the assumption {\rm\textbf{(A.2)}}, there exist $\varrho,h_{\de_1}>0$ such that 
\begin{align}\label{eq cont a}
	\sup_{\substack{x,y\in B_R(x_0)\\ |x-y|\leq \varrho}} |a(x+h,y+h)-a(x,y)|\leq \de_1 \quad\mbox{for all }h\in B_{h_{\de_1}}(0).
\end{align}
Fix $z\in B_\frac{R}{2}(x_0)$ and $r_z\in (0,1)$ small enough such that $r_z\leq \min\{\frac{R}{2},\frac{\varrho}{2},h_{\de_1}\}$ and $B_{r_z}(z)\subset B_{R/2}(x_0)$. Then, for all $x,y\in B_{r_z}(z)$, we see that $(z-y)$ and $(z-x)\in B_{h_{\de_1}}(0)$. Therefore,
\begin{align*}
	&\sup_{x,y\in B_{r_z}(z)} |a(x-y+z,z)-a(x,y)|\leq\de_1 \quad\mbox{and}\\
	&\sup_{x,y\in B_{r_z}(z)} |a(z,y-x+z)-a(x,y)|\leq\de_1.
\end{align*}
Consequently, using the symmetry of $a(\cdot,\cdot)$, we see that
\begin{align*}
	a_z(x,y):=\frac{1}{2}\big(a(x-y+z,z)+a(y-x+z,z)\big)
\end{align*}
is translation invariant in $B_{r_z}(z)\times B_{r_z}(z)$ and satisfies 
\begin{align*}
	\|a-a_z\|_{L^\infty(B_{r_z}(z)\times B_{r_z}(z))}\leq \de_1.
\end{align*}
Defining
\begin{align*}
	\tl a(x,y):=\begin{cases}
		a_z(x,y) &\mbox{if }(x,y)\in B_{r_z}(z)\times B_{r_z}(z),\\
		a(x,y) &\mbox{if }(x,y)\notin B_{r_z}(z)\times B_{r_z}(z),
	\end{cases}
\end{align*}
it follows that $\tl a(\cdot,\cdot)$ satisfies the assumption {\rm\textbf{(A.1)}} in $B_{r_z}(z)\times B_{r_z}(z)$ and 
\begin{align}\label{eqia64}
	\|a-\tl a\|_{L^\infty(\mb R^N\times \mb R^N)}\leq \de_1.
\end{align}
We set \begin{align*}
	u_{z}(x):=\frac{1}{\mc M}u(r_zx+z) \quad\mbox{for }x\in B_2.
\end{align*}
By the scaling property of the fractional operators, it is easy to observe that $u_{z}$ solves
\begin{align*}
	(-\De)_p^{s_1} u_{z}+(-\De)_{q,a_1}^{s_2}u_{z}=f_z \quad\mbox{in }B_2,
\end{align*}
where $a_1(x,y)=\mc M^{q-p}r_z^{s_1p-s_2q}a(r_zx+z,r_zy+z)$ and $f_z(x):=\frac{r_z^{ps_1}}{\mc M^{p-1}}f(r_zx+z)$.
Moreover, since $z\in B_{R/2}(x_0)$, we see that 
\begin{align*}
	|y-z|\geq |y-x_0|\Big( 1-\frac{|z-x_0|}{|y-x_0|} \Big)\geq \frac{|y-x_0|}{2}\quad\mbox{for all }y\in B_{R_0/2}(x_0)^c.
\end{align*}
Consequently, 
\begin{align}\label{eq app scl}
	&T_{s\el}^\el(u;z,r_z) \nonumber\\
	&\leq r_z^{\el s} \left[2^{N+\el s} \int_{B_\frac{R_0}{2}(x_0)^c}\frac{|u(y)|^{\el-1}\,dy}{|y-x_0|^{N+\el s}} + \|u\|_{L^\infty(B_\frac{R_0}{2}(x_0))}^{\el-1} \int_{B_{r_z}(z)^c} \frac{dy}{|y-z|^{N+\el s}}\right] \nonumber\\
	&\leq 2^{N+\el s} T_{s\el}^\el(u;x_0,R_0/2) + \Big(\frac{N|B_1|}{\el s}\Big) \|u\|_{L^\infty(B_{R_0/2}(x_0))}^{\el-1}.
\end{align}
Recalling the definition of $\mc M$ and using \eqref{eq app scl},  we have
\begin{align*}
	\|u_{z}\|_{L^\infty(B_1)}\leq 1, \quad \sum_{(\el,s)}T^{\el}_{\el s}(u_z;0,1)^{1/(\el-1)}\leq 1\quad \mbox{and}\quad\|f_z\|_{L^\Bbbk(B_1)}\leq\de_1.
\end{align*}
Furthermore, for $\tl a_1(x,y):=\mc M^{q-p}r_z^{s_1p-s_2q}\tl a(r_zx+z,r_zy+z)$, there holds
\begin{align*}
	\|a_1-\tl a_1\|_{L^\infty(\mb R^N\times\mb R^N)}= r_z^{s_1p-s_2q}\mc M^{q-p}\|a-\tl a\|_{L^\infty(\mb R^N\times\mb R^N)}\leq \de_1,
\end{align*}
where we have used \eqref{eqia64} and \eqref{eqia65}. Therefore,  from Proposition \ref{propap2}, we obtain
\begin{align*}
	[u_{z}]_{C^{0,{\Theta_1}-\e}({B_{1/8}})} \leq C_1(N,p,q,s_1,s_2,\Bbbk,\|f\|_{L^\Bbbk},\|a\|_\infty,\e).
\end{align*}
%By scaling back, we get
%\begin{align*}
%	[u]_{C^{0,{\Theta_1}-\e}(B_{r_z/2}(z))} 
%	\leq \frac{C_1}{r_z^{{\Theta_1}-\e}} \mc M.
%\end{align*}
By scaling back to $u$ and since $z\in B_{R/2}(x_0)$ was arbitrary, by a standard covering argument, we can complete the proof of the theorem. 
%Finally, recalling the definition of $\mc M$ and noting the dependency of the constant $\de_1$, we conclude \eqref{cont bd st}. 
\begin{Remark}
	It is clear from the above proof that for $\Bbbk=\infty$, Theorem \ref{thmingenk} holds for the case $ps_1=qs_2$, too. Indeed, in this case, we approximate only the modulating coefficient $a(\cdot,\cdot)$ in Lemma \ref{lemstab} by keeping $f$ fixed. In the proof of Theorem \ref{thmingenk}, we proceed with $\mc M(R):=\mc C\|u\|_{L^\infty(B_{R})} +1+\sum_{(\el,s)}2^\frac{N+s\el}{\el-1} T^{\el}_{\el s}(u;x_0,R)^\frac{1}{\el-1},$ instead of $\mc M(R_0)$ and take ${\de_1\mc M(R)^{p-q}}$ in place of $\de_1$ in \eqref{eq cont a}. 
\end{Remark}

\subsection{Boundary regularity}
We define the distance function as $d(x):={\rm dist}(x,\mb R^N\setminus\Om)$. Then, we have the following estimate
\begin{Proposition}\label{propupper}
	Let $\Om\subset\mb R^N$ be a bounded domain with $C^{1,1}$ boundary $\pa\Om$ and for $s_1<q's_2$, we assume that $a(\cdot,\cdot)\in C^{0,\ba}(\mb R^N\times\mb R^N)$ with $\ba>s_2$.
	% Under the hypothesis of Lemma \ref{lemdist}, 
	Let $u\in \mc W(\mb R^N)\cap L^\infty_{\rm loc}(\Om)$ be such that $|\mc Lu| \leq K$, weakly in $\Om$, for some $K>0$. Then, for all $0<\sg<s_1$, there exists a constant $\mc C_u>0$ (depending only on data of the problem and $\|u\|_{L^\infty_{\rm loc}(\Om)}$) such that
	\[ |u| \leq \mc C_u \;d^{\sg} \quad\mbox{in } \Om. \]
\end{Proposition}
\begin{proof}
	We first obtain a lower bound on the quantity $(-\De)_{q,a}^{s_2} d^{\al}$, for all $\al\in [s_2,1)$.  Indeed, proceeding as in  \cite[Lemma 3.12]{JDSacv}, the case $\al>q's_2$ is exactly the same as Case (i) of the proof presented there, that is,
	\begin{align}\label{eqi80}
		(-\De)_{q,a}^{s_2} \ d^\al = g \in L^\infty(\Om_\rho) \quad\mbox{weakly in }\Om_\rho,
	\end{align} 
	where $\Om_\rho:=\{x\in\Om \  : \ d(x)<\rho\}$.
	For the case $\al< q's_2$, a similar procedure to that of \cite[Lemma 3.12]{JDSacv} yields that 
	\begin{align}\label{eqi81}
		(-\De)_{q,\tl a(\cdot)}^{s_2} \ d^\al \geq C_1 d^{\al(q-1)-qs_2}  \quad\mbox{weakly in }\Om_\rho,
	\end{align} 
	where $\tl a(x)=a(x,x)$ and $C_1$ is a positive constant. Furthermore, proceeding with the same notation as in \cite[Lemma 3.12]{JDSacv},
	the integrand $\tl I_{2,\e}(X)$ is estimated as below: 
	\begin{align*}%\label{eq121}
		\tl I_{2,\e}(X) 
		&= \int_{B_\e(X)^c\cap B_{2\rho}} a(\Phi(X),\Phi(X)) \frac{[(X_N)_+^{\al}-(Y_N)_+^{\al}]^{q-1}}{|\na\Phi(X)(X-Y)|^{N+qs_2}} J_\Phi(X) dY \nonumber\\
		&\,\, + \int_{B_\e(X)^c\cap B_{2\rho}} \big[a(\Phi(X),\Phi(Y)) -a(\Phi(X),\Phi(X)) \big] \nonumber\\ &\qquad\quad\times\frac{[(X_N)_+^{\al}-(Y_N)_+^{\al}]^{q-1}}{|\na\Phi(X)(X-Y)|^{N+qs_2}} J_\Phi(X) dY \nonumber\\
		&=: I_{5,\e}(X)+ I_{6,\e}(X). \end{align*} 
	Using the $\ba$-H\"older continuity of $a(\cdot,\cdot)$ and the Lipschitz nature of $\Phi$ and $\na\Phi$, we obtain
	\begin{align*}
		|I_{6,\e}(X)|\leq C\int_{B_\e(X)^c\cap B_{2\rho}} \frac{1}{|X-Y|^{N+qs_2-\al(q-1)-\ba}}dY=C_\rho<\infty,
	\end{align*}
	where we have used the fact that $\ba>qs_2-\al(q-1)$. The term $I_{5,\e}(X)$ is handled as in the estimate to prove \eqref{eqi81}. This concludes that, for $\al<q's_2$, 
	\begin{align}\label{eqi82}
		(-\De)_{q,a}^{s_2} \ d^\al \geq C_2 d^{\al(q-1)-qs_2}  \quad\mbox{weakly in }\Om_\rho.
	\end{align} 
	Also, note that $d^\al\in W^{s_1,p}(\Om')\cap W^{s_2,q}(\Om')$, for some $\Om\Subset\Om'$ and for all $s_2-1/q<\al<1$. Consequently, proceeding similar to  \cite[Proposition 3.14]{JDSacv} and taking into account the local boundedness of $u$ together with \eqref{eqi80} and \eqref{eqi82}, we can complete the proof of the proposition. 
\end{proof}
Finally, we have our global H\"older continuity result.\\
\textbf{Proof of Corollary \ref{corbdry}}:  The proof can be completed by standard covering arguments and using the interior regularity result of Theorem \ref{thmingenk} (note that $\Bbbk=\infty$, here) together with the boundary estimate of Proposition \ref{propupper}.

\begin{appendix}
	
 \section{On the existence result}
	\renewcommand{\theequation}{E.\arabic{equation}}
	
	% reset the counter
	\setcounter{equation}{0}
Here, we give the functional setup and a basic framework for the existence result for the following problem:
	\begin{equation*}
		\left\{ \begin{array}{rl}
			\pa_t u+ (-\De)_p^{s_1}u+(-\De)_{q,a}^{s_2}u&=f \quad\mbox{in }\Om\times I,\\
			u&=g \quad\mbox{in }(\mb R^N\setminus \Om)\times I,\\
			u(\cdot,t_0)&=u_0(\cdot) \quad\mbox{in }\Om,
		\end{array}
		\right. \tag{$GNP$}\label{genPrb}
	\end{equation*}
 where $I=[t_0,t_1]$, $u_0\in L^2(\Om)$, $f\in \big(\mc Y(I,\mc W(\Om'))\big)^*$, $g\in \mc Y(I,\mc W(\Om'))\cap L^{p-1}(I;L^{p-1}_{s_1p}(\mb R^N))\cap L^{q-1}(I;L^{q-1}_{s_2q,a}(\mb R^N))$ and $\pa_t g\in \big(\mc Y(I,\mc W(\Om'))\big)^*$, for $\Om\Subset\Om'\subset\mb R^N$.
	\begin{Definition}
	We say that a function $u\in \mc Y(I,\mc W(\Om'))\cap C(I;L^2(\Om))\cap L^{p-1}(I;L^{p-1}_{s_1p}(\mb R^N))\cap L^{q-1}(I;L^{q-1}_{s_1q,a}(\mb R^N))$ is a weak solution to  \eqref{genPrb} if the following hold:
		\begin{itemize}
			\item[(i)] $u\in X_{\rm g(t)}(\Om,\Om')$ for a.e. $t\in I$, where ${\rm g(t)}(x)=g(x,t)$;
			\item[(ii)] $u(\cdot,t)\to u_0$ in $L^2(\Om)$, as $t\to t_0$;
			\item[(iii)] 	for all $J=[T_0,T_1]$ and every $\phi\in \mc Y(J;X_0(\Om,\Om'))\cap C^1(J; L^2(\Om))$
			\begin{align*}%\label{eqWFG}
				&-\int_{J}\int_{\Om} u(x,t)\pa_t \phi(x,t)dxdt \\
				&\quad+ \int_{J} \iint_{\mb R^{2N}} \frac{[u(x,t)-u(y,t)]^{p-1}}{|x-y|^{N+s_1p}}(\phi(x,t)-\phi(y,t))dxdydt \nonumber\\
				&\quad+\int_{J} \iint_{\mb R^{2N}}a(x,y) \frac{[u(x,t)-u(y,t)]^{q-1}}{|x-y|^{N+s_2q}}(\phi(x,t)-\phi(y,t))dxdydt \nonumber \\
				&= \int_{\Om} u(x,T_0) \phi(x,T_0)dx- \int_{\Om} u(x,T_1) \phi(x,T_1)dx+  \langle f(\cdot,\cdot),\phi(\cdot,\cdot)\rangle_{\mc Y, \mc Y^*}.
			\end{align*}
		\end{itemize}
	\end{Definition}
\noindent	Proceeding similar to \cite[Proposition 1.2, Chapter III]{showalter}, we can show that 
	\begin{align*}
		\mb W(I):= \{v\in \mc Y(I;\mc W(\Om)) \ : \ u^\prime\in \big(\mc Y(I,\mc W(\Om))\big)^*\}\subset C(I;L^2(\Om)).
	\end{align*}
	Additionally, for $v\in \mb W(I)$,
	\begin{align*}
		t\mapsto \|v(t)\|_{L^2(\Om)}^2 \mbox{ is absolutely continuous and } \frac{d}{dt}\|v(t)\|_{L^2(\Om)}^2=2\langle v^\prime(t), v(t)\rangle.
	\end{align*}
	We have the following existence result.
	\begin{Theorem}\label{thm exst Appn}
		Let $p\geq 2$ and suppose that 
		\begin{align*}
			\lim_{t\to t_0} \| g(\cdot,t)-g_0 \|_{L^2(\Om)}=0 \quad\mbox{for some }g_0\in L^2(\Om).
		\end{align*}
		Then, there exists a unique solution to problem \eqref{genPrb}.
	\end{Theorem}
	\begin{proof}
	The proof is similar to \cite[Theorem A.3]{brascoP}. Here, we highlight only the key steps, for instance, boundedness and coercivity of the operators involved in the proof.	For a.e. $t\in I$, we define  $\mc A_t: X_{{\rm g}(t)}(\Om,\Om')\to (\mc W(\Om'))^*$ by 
		\begin{align*}
			\langle \mc A_t(v),\phi\rangle =
			%& \int_{\Om'}\int_{\Om'} \frac{[v(x)-v(y)]^{p-1}}{|x-y|^{N+s_1p}}(\phi(x)-\phi(y))dxdy +2\int_{\Om}\int_{\mb R^N\setminus\Om'} \frac{[v(x)-g(y,t)]^{p-1}}{|x-y|^{N+s_1p}}\phi(x)dxdy  \nonumber\\
			&\sum_{(\el,s)}\int_{\Om'}\int_{\Om'} a_\el(x,y) \frac{[v(x)-v(y)]^{\el-1}}{|x-y|^{N+s\el}}(\phi(x)-\phi(y))dxdy \nonumber\\
			&+2 \sum_{(\el,s)}\int_{\Om}\int_{\mb R^N\setminus\Om'} a_\el(x,y) \frac{[v(x)-g(y,t)]^{\el-1}}{|x-y|^{N+s\el}}\phi(x)dxdy \nonumber\\
			&=:	\langle \mc A^p_t(v),\phi\rangle
			+	\langle \mc A^q_t(v),\phi\rangle.
		\end{align*}
		%	On account of \cite[Remark 1]{korvenpaa}, we have
		%	\begin{align}\label{eq120}
			%		|\langle \mc A^p_t(v),\phi\rangle|\leq C \big( \|u\|^{p-1}_{W^{s_1,p}(\Om')} + T^{p}_{ps_1}(g;z,r) \big) \|v\|_{W^{s_1,p}(\Om')},	
			%	\end{align}
		%	where $z\in \Om$ and $r={\rm dist}(\Om,\pa\Om')>0$. Furthermore,
	Applying H\"older's inequality and recalling the definition of $W$ (stated as in \eqref{eqw}), we obtain
		\begin{align}\label{eq121}
			&|\langle \mc A^q_t(v),\phi\rangle| \nonumber\\&\leq \int_{\Om'}\int_{\Om'} a(x,y) \frac{|v(x)-v(y)|^{q-1}}{|x-y|^{N+s_2q}}|\phi(x)-\phi(y)|dxdy \nonumber\\
			&\quad +c(q) \int_{\Om}\int_{\mb R^N\setminus\Om'} a(x,y) \frac{|v(x)|^{q-1}+|g(y,t)|^{q-1}}{|x-y|^{N+s_2q}}|\phi(x)|dxdy  \nonumber\\
			%&\leq C [v]^{q-1}_{W^{s_2,q}_{a}(\Om')} [\phi]_{W^{s_2,q}_{a}(\Om')}+ C\int_{\Om}W(x)|v(x)|^{q-1}|\phi(x)|dx \nonumber\\
			%&\quad+C\int_{\Om}|\phi(x)|dx\int_{\mb R^N\setminus\Om'} a(x,y)\frac{|g(y,t)|^{q-1}}{|x-y|^{N+qs_2}}dy \nonumber\\
			&\leq c [v]^{q-1}_{W^{s_2,q}_{a}(\Om')} [\phi]_{W^{s_2,q}_{a}(\Om')} + c\Big(\int_{\Om}W(x)|v(x)|^{q}dx \Big)^\frac{1}{q'} \Big(\int_{\Om}W(x)|\phi(x)|^{q}dx \Big)^\frac{1}{q} \nonumber\\
			&\quad+ c \Big(\int_{\Om}|\phi(x)|^{p}dx \Big)^\frac{1}{p}  T^{q}_{qs_2,a}(g;z,r)^{p/(p-1)} \nonumber\\
			&\leq c\big(\|v\|^{q-1}_{\mc W(\Om')}+ T^{q}_{qs_2,a}(g;z,r)^{p/(p-1)}\big) \|\phi\|_{\mc W(\Om')},
		\end{align}
	 where $z\in \Om$ and $r={\rm dist}(\Om,\pa\Om')>0$. An analogous result holds for the $p$-term, too. Consequently,  $\mc A_t$ is a well defined operator. 
		%Moreover, by the property of $(-\De)_{\el}^s$, it is not difficult to show that the operator $\mc A_t$ is monotone.
		Next, we define $\mc A: X_0(\Om,\Om')\to (\mc W(\Om'))^*$ by 
		\begin{align*}
			\mc A(v,t)=\mc A_t(v+{\rm g}(t)).
		\end{align*}
	Then, the operator $\mc A$ is monotone and hemicontinuous (see \cite[Lemma 3]{korvenpaa}).
	On account of H\"older's inequality and the definition of norm on $\mc W(\Om')$, it is easy to observe that 
		\begin{align}\label{eq81}
			&\int_{\Om'}\int_{\Om'} \frac{[v(x)+g(x,t)-v(y)-g(y,t)]^{p-1}}{|x-y|^{N+s_1p}}(\phi(x)-\phi(y))dxdy \nonumber\\
			&+ \int_{\Om'}\int_{\Om'} a(x,y) \frac{[v(x)+g(x,t)-v(y)-g(y,t)]^{q-1}}{|x-y|^{N+s_2q}}(\phi(x)-\phi(y))dxdy \nonumber\\
			&\leq C \Big(\|v\|_{W^{s_1,p}(\Om')}^{p-1}+ \|{\rm g}(t)\|_{W^{s_1,p}(\Om')}^{p-1}+ \|v\|_{W^{s_2,q}_{a}(\Om')}^{q-1}+ \|{\rm g}(t)\|_{W^{s_2,q}_{a}(\Om')}^{q-1}\Big) \nonumber\\
			&\qquad\times \|\phi\|_{\mc W(\Om')}.
		\end{align}
		Moreover, proceeding similar to \eqref{eq121}, we can prove that
		\begin{align}\label{eq82}
			&\sum_{(\el,s)}\int_{\Om}\int_{\mb R^N\setminus\Om'}a_\el(x,y) \frac{[v(x)+g(x,t)-g(y,t)]^{\el-1}}{|x-y|^{N+s\el}}\phi(x)dxdy \nonumber\\
			%&+ \int_{\Om}\int_{\mb R^N\setminus\Om'} a(x,y) \frac{[v(x)+g(x,t)-g(y,t)]^{q-1}}{|x-y|^{N+s_2q}}\phi(x)dxdy \nonumber \\
			%&\leq C \int_{\Om} \big( |v(x)|^{p-1}+|g(x,t)|^{p-1})|\phi(x)|dx+ C \Bigg(\int_{\mb R^N\setminus\Om'}\frac{|g(y,t)|^{p-1}}{1+|y|^{N+s_1p}}dy \Bigg)\int_{\Om}|\phi(x)|dx \nonumber\\
			%& + C \int_{\Om} W(x)\big( |v(x)|^{q-1}+|g(x,t)|^{q-1})|\phi(x)|dx+ C \Bigg(\int_{\mb R^N\setminus\Om'}a(x,y)\frac{|g(y,t)|^{q-1}}{1+|y|^{N+s_2q}}dy \Bigg)\int_{\Om}|\phi(x)|dx \nonumber\\
			&\leq C\Big(\|v\|_{\mc W(\Om')}^{p-1}+ \|{\rm g}(t)\|_{\mc W(\Om')}^{p-1}+ \|{\rm g}(t)\|_{L^{p-1}_{s_1p}(\mb R^N)}^{p-1} \Big) \|\phi\|_{\mc W(\Om')} \nonumber\\
			&\quad+ C\Big(\|v\|_{\mc W(\Om')}^{q-1}+ \|{\rm g}(t)\|_{\mc W(\Om')}^{q-1}+\|{\rm g}(t)\|_{L^{q-1}_{s_2q,a}(\mb R^N)}^{q-1} \Big) \|\phi\|_{\mc W(\Om')}.
		\end{align}
		Coupling \eqref{eq81} and \eqref{eq82}, we obtain
		\begin{align*}
			|\langle \mc A(v,t),\phi\rangle| &\leq C \Big(\|v\|_{\mc W(\Om')}^{p-1}+ \|{\rm g}(t)\|_{\mc W(\Om')}^{p-1}+ \|{\rm g}(t)\|_{L^{p-1}_{s_1p}(\mb R^N)}^{p-1} \nonumber\\
			&\quad+ \|v\|_{\mc W(\Om')}^{q-1}+ \|{\rm g}(t)\|_{\mc W(\Om')}^{q-1}+\|{\rm g}(t)\|_{L^{q-1}_{s_2q,a}(\mb R^N)}^{q-1} \Big) \|\phi\|_{\mc W(\Om')},
		\end{align*}
		which proves the boundedness of $\mc A$.
		Next, to obtain coercivity of the operator $\mc A$, 
		for any $v\in X_0(\Omega,\Omega')$, using \eqref{eqmon} and Young's inequality, we first see that
		\begin{align*}%\label{eq114}
			&\int_{\Om'}\int_{\Om'}a(x,y) [v(x)+g(x,t)-v(y)-g(y,t)]^{q-1}(v(x)-v(y))d\mu_2	\nonumber\\
			&=
			\int_{\Omega'}\int_{\Omega'} a(x,y)   \Big([v(x)+g(x,t)-v(y)+g(y,t)]^{q-1}-[g(x,t)-g(y,t)]^{q-1} \Big) \nonumber\\
			&\qquad\quad\times (v(x)-v(y))d\mu_{2} \nonumber\\
			& \ +\int_{\Omega'}\int_{\Omega'}   a(x,y) [g(x,t)-g(y,t)]^{q-1}  (v(x)-v(y))d\mu_{2} \nonumber\\
			&\geq c_q \int_{\Omega'}\int_{\Omega'} a(x,y)|v(x)-v(y)|^q d\mu_2 \nonumber\\
			& \ - c_q \int_{\Omega'}\int_{\Omega'} a(x,y)  |g(x,t)-g(y,t)|^{q-1}  |v(x)-v(y)|d\mu_{2} \nonumber\\
			&\geq c\int_{\Omega'}\int_{\Omega'} |v(x)-v(y)|^q d\mu_2 - C \int_{\Omega'}\int_{\Omega'}   |g(x,t)-g(y,t)|^{q}  d\mu_{2}.
		\end{align*}
		Proceeding similarly, and using the definition of $W$, we observe that 
		\begin{align*}%\label{eq116}
			&\int_{\Omega}\int_{\mathbb R^N\setminus\Omega'} a(x,y) [v(x)+g(x,t)-g(y,t)]^{q-1}v(x)d\mu_2 \nonumber\\
			%&= \int_{\Omega}\int_{\mathbb R^N\setminus\Omega'} a(x,y) \Big([v(x)+g(x,t)-g(y,t)]^{q-1}-[g(x,t)-g(y,t)]^{q-1}\Big)v(x)d\mu_2 \nonumber\\
			%& \ + \int_{\Omega}\int_{\mathbb R^N\setminus\Omega'} a(x,y) [g(x,t)-g(y,t)]^{q-1}v(x)d\mu_2 \nonumber\\
			&\geq c(q) \int_{\Omega}\int_{\mathbb R^N\setminus\Omega'}  \frac{a(x,y)}{|x-y|^{N+qt}}|v(x)|^{q} dxdy \nonumber\\
			& \ - c(q) \int_{\Omega}\int_{\mathbb R^N\setminus\Omega'} \frac{a(x,y)}{|x-y|^{N+qs_2}} (|g(x,t)|^{q-1}+|g(y,t)|^{q-1})|v(x)|dxdy \nonumber\\
			%&\geq c \int_{\Omega'} W(x)|v(x)|^q dx-C\Big(\int_{\Omega'}W(x)|g(x,t)|^{q}dx \Big)^\frac{1}{q'} \Big(\int_{\Omega'}W(x)|v(x)|^{q}dx \Big)^\frac{1}{q} \nonumber\\ 
			%& \ - C \Big(\int_{\Omega}|v(x)|dx \Big)  T_{qs_2,a}(g;z,r)^{q-1} \nonumber\\
			&\geq c \int_{\Omega'} W(x)|v(x)|^q dx-C\int_{\Omega'}W(x)|g(x,t)|^{q}dx -\epsilon \int_{\Omega'}|v(x)|^pdx \nonumber\\
			&\quad- C T^q_{qs_2,a}(g;z,r)^\frac{p}{p-1},
		\end{align*}
		where we have also used the fact that $v=0$ in $\Omega'\setminus\Omega$ and Young's inequality on the last line. Analogous results hold for the corresponding $p$-counterparts, too. Consequently,  recalling  the definition of the norm on $\mathcal{W}(\Omega')$, we obtain
		\begin{align*}%\label{eq118}
			\langle  \mathcal A(v),v\rangle &\geq c \min\{ \|v\|_{\mathcal{W}(\Omega')}^{p}, \|v\|_{\mathcal{W}(\Omega')}^{q} \} -C \|{\rm g}(t)\|_{W^{s,p}(\Omega')}^{p}-C\|{\rm g}(t)\|_{W^{t,q}_{b}(\Omega')}^{q} \nonumber\\
			& \ -CT^p_{ps}({\rm g}(t);z,r)- C T^q_{qs_2,a}({\rm g}(t);z,r)^{p/(p-1)},
		\end{align*}
		where the constants $c$ and $C$ are independent of $v$, and hence the operator $\mathcal{A}$ is coercive. 
		% Since $g\in C(I;L^2(\Om))$, we define $g_0={\rm g}(t_0)$ in $L^2(\Om)$. Therefore, for all $u_0\in L^2(\Om)$, we obtain a unique solution $v\in\mb W(I)$ to the problem
		%	\begin{align*}
			%		v^\prime(t)+\mc A(v,t)=-{\rm g}^\prime(t)+f(t) \quad\mbox{in } \big(\mc Y(I;X_0(\Om,\Om'))\big)^*, \quad v(t_0)=u_0-g_0.
			%	\end{align*}
		% Further, $v\in C(I; L^2(\Om))$ and for all $\phi\in \mc Y(I;X_0(\Om,\Om'))$,
		%	\begin{align*}
			%		\int_{I}\langle v^\prime(t)+{\rm g}^\prime(t),\phi(t)\rangle dt  + \int_{I} \langle\mc A_t(v(t)+{\rm g}(t)),\phi(t)\rangle dt=   \langle f(\cdot,\cdot),\phi(\cdot,\cdot)\rangle_{\mc Y, \mc Y^*}.
			%	\end{align*}
		%	Setting $u=v+g$, we observe that 
		%	\begin{align*}
			%		&u\in \mc Y(I,X_{{\rm g}(\cdot)}(\Om,\Om'))\cap L^{p-1}(I;L^{p-1}_{s_1p}(\mb R^N))\cap L^{q-1}(I;L^{q-1}_{s_1q,a}(\mb R^N))\cap C(I;L^2(\Om)) \mbox{ with}\\
			%		& \pa_t u\in \big(\mc Y(I;X_0(\Om,\Om')) \big)^*
			%	\end{align*}
		%	and 
		%	\begin{align*}
			%		\int_{I}\langle \pa_t u,\phi(t)\rangle dt  + \int_{I} \langle\mc A_t(u),\phi(t)\rangle dt=   \langle f(\cdot,\cdot),\phi(\cdot,\cdot)\rangle_{\mc Y, \mc Y^*}.
			%	\end{align*}
		% Then, for all $J=[T_1,T_2]$ and if $\phi\in \mc Y(J;X_0(\Om,\Om'))\cap C^1(J;L^2(\Om))$, using integration by parts we can show that $u$ satisfies \eqref{eqWFG}, that is, $u$ a weak solution to problem \eqref{genPrb} (for details, see the proof of \cite[Theorem A.3]{brascoP}). 
	\end{proof}
	%
	%\begin{Remark}
	% Suppose that $2\leq p\leq q\leq p\big(1+\frac{ps_1}{N}\big)$. Let $u$ be a weak solution to problem \eqref{genPrb} (or a local weak solution to problem \eqref{probM}) with $f\in L^\infty_{\rm loc}(\Om\times I)$ such that $u\in L^\infty_{{\rm loc}}(I;L^{p-1}_{s_1p}(\mb R^N))\cap L^\infty_{{\rm loc}}(I;L^{q-1}_{s_2q,a}(\mb R^N))$.	Then,  using the Caccioppoli inequality of Lemma \ref{cacciopp} (note that local boundedness assumption on $u$ is not required there for $q\leq p_*({s_1})$, see \cite[Lemma 4.1]{byun}) and proceeding similar to the proof of \cite[Theorem 1]{ding}, we can prove that $u\in L^\infty_{{\rm loc}} (\Om\times I)$.
	%\end{Remark}
\end{appendix}

\medskip
{\bf Funding:}
The first author is funded by IFCAM (Indo-French Centre for Applied Mathematics) IRL CNRS 3494.

\end{document}